\documentclass[reqno]{amsart}
\usepackage{amssymb,color}
\usepackage{amsmath}
\usepackage{amsfonts}
\usepackage{fouridx}

\usepackage[normalem]{ulem} 
\usepackage{soul} 

\DeclareMathOperator*{\esssup}{\mathrm{ess\,sup}}

\newcommand{\dist}{\mbox{dist}}

\allowdisplaybreaks[3]

\newcommand\delc[1]{}

\newcommand\comcd[1]{}

\newcommand\del[1]{}

\newcommand\deln[1]{}
\newcommand\delr[1]{}

\newcommand\comad[1]{}

\newcommand\comr[1]{{\color{red} #1}}

\newcommand\Greendel[1]{}
\newcommand\old[1]{}

\numberwithin{equation}{section}
\newcommand{\embed}{\hookrightarrow}

\newcommand{\lb}{\langle}
\newcommand{\rb}{\rangle}

\newcommand\ex[1]{\operatorname{e}^{#1}}

\newcommand{\then}{\Longrightarrow}

\def\E{{\mathbb E}\,}
\def\ds{\displaystyle}
\def\vs{\vspace{.1cm}}

\newcommand{\rA}{\mathrm{A}}
\newcommand{\rB}{\mathrm{B}}

\newcommand{\rH}{\mathrm{ H}}
\newcommand{\rK}{\mathrm{K}}
\newcommand{\rV}{\mathrm{ V}}
\newcommand{\rU}{\mathrm{ U}}

\newcommand{\cH}{\mathcal{ H}}

\newcommand\todown{\searrow}

\newcommand{\eps}{\varepsilon}

\renewcommand{\a}{\alpha}

\renewcommand{\d}{\delta}

\def\d{{\,\rm d}}
\def\old#1{}

\def\text#1{{\rm #1}}
\def\newold#1{}

\def\divv{{\rm div}\,}

\def\vp{\varphi}

\theoremstyle{plain}
\newtheorem{theorem}{Theorem}[section]
\theoremstyle{remark}
\newtheorem{remark}[theorem]{Remark}

\theoremstyle{plain}

\newtheorem{lemma}[theorem]{Lemma}
\newtheorem{proposition}[theorem]{Proposition}
\newtheorem{definition}[theorem]{Definition}

\newtheorem{assumption}[theorem]{Assumption}

\numberwithin{equation}{section}

\begin{document}

\baselineskip 12pt

\title[LDP for invariant measures of 2D SNSE]
{Large deviations principle for the invariant measures of the 2D stochastic Navier-Stokes equations on a torus }
\author{Z. Brze{\'z}niak}
\address{Department of Mathematics\\
The University of York\\
Heslington, York YO10 5DD, UK} \email{zdzislaw.brzezniak@york.ac.uk}
\author{S. Cerrai}
\address{Department of Mathematics\\
University of Maryland\\
College Park, MD, 20742,  USA} \email{cerrai@math.umd.edu}
\date{\today}




\begin{abstract}
We prove here the validity of a large deviation principle for the family of invariant measures associated to a two dimensional Navier-Stokes equation on a torus, perturbed by a smooth additive noise.
\end{abstract}

\maketitle \tableofcontents
\section{Introduction}\label{sec-intro}
In the present paper we are dealing with  2-D   Navier Stokes equations with periodic boundary conditions, perturbed by a small additive noise. These boundary conditions are usually realized by considering the problem on a two-dimensional torus $\mathbb{T}^2$, see Section \ref{sec-prel} for more details. To fix readers attention, let us write down these equations in a functional form, as
\begin{equation}
\label{eqn-SNSE-intro}
du(t)+\rA u(t)\, dt+\rB(u(t),u(t))\, dt=\sqrt{\eps}\,\,dw(t),\ \ \ \ u(0)=u_0,
\end{equation}
for $0<\eps<<1$.

Full definitions of the symbols involved can be found later in Section \ref{sec-prel}, but, for the time being, let us recall that $\rA$ is the Stokes operator, equal, roughly speaking, to the Laplace operator (acting on vector fields) composed with the Leray-Helmoltz projection $P$, defined  on the space of zero mean and  square integrable vector fields with values in  the subspace $\rH$ of  divergence free vector fields, the convection $\rB(u,u)$ is equal to $P (u\nabla u)$,  $w(t)$ is a $\rK$-cylindrical Wiener process, for $\rK=D(\rA^{\frac\alpha2})$ with $\alpha>1$, and $u_0 \in \rH$. Of course, because $P$ nullifies the gradients, the gradient of the pressure term $\nabla p$ disappears in such a formulation. Basic questions about such a problem are now well understood, and we simply refer to the papers \cite{Flandoli_1994} and \cite{Brz+Li_2006}  and to the chapter 15 of the monograph \cite{DaPrato+Z_1996}.

It is know that, for every fixed $\eps>0$, the Markov process on $\rH$ generated by equation \eqref{eqn-SNSE-intro} has an invariant measure $\mu_\eps$ (see \cite{Flandoli_1994}), which is also unique and ergodic (see \cite{Ferrario_1999} and also \cite{Hairer+Mattingly_2006}).
The objective of our paper is study of the validity of a large deviation principle (LDP) for the family of invariant measures $\{\mu_\eps\}_{\eps>0}$. To be more precise, our purpose is to  show   that the family of probability measure$\big\{\mu_\eps \big\}_{\eps >0}$ satisfies a LDP, as $\eps\downarrow 0$, with rate $\eps$ and action functional  equal to the quasi-potential $\rU$ associated to the  controlled deterministic NSE, also known as  the skeleton equation,
\begin{equation}
\label{eqn-SNSE-control}
u^\prime(t)+\rA u(t)+\rB(u(t),u(t))=f(t),\ \ \ \ u(0)=u_0,
\end{equation}
where $f\in L^2(0,\infty;D(\rA^{\frac\alpha2}))$. The quasi-potential $\rU(v)$, for $v \in \rH$, can be defined as the infimum, over all $T>0$,  of the energy of the control $f$, with respect to the norm of the reproducing kernel Hilbert space $\rK$ of the law $\mathcal{L}(W(1))$, i.e.
\[\frac12 \int_0^T \vert \rA^{\frac\alpha2}  f(t)\vert_{\rH}^2,\]
 such that the solution $u$ to the skeleton equation \eqref{eqn-SNSE-control}, with initial data $u(0)=0$, reaches the  state $v$ at time $T$, i.e. $u(T)=v$. To this purpose, we refer to equation \eqref{eqn-B4-QP} for a version of the definition of $\rU$ using both positive and negative times and \eqref{eqn-A1} for a representation of $\rU$ using the skeleton equation over the negative half-line $(-\infty,0]$.

The quasi-potential $\rU$ was an important object in our recent study \cite{BCF_2013} with M. Freidlin and in some sense our current paper is a natural continuation of that work. The two other works on which we depend a lot in our investigation is the paper \cite{Cerrai+Roeckner_2005} by the second named authour and M. R\"ockner and \cite{sowers} in which a similar question was investigated for reaction diffusion equation with polynomially bounded, resp. bounded, reaction term.

Let us make an important comment about the assumption $\alpha>1$. In fact,  the Markov process on $\rH$, generated by problem \eqref{eqn-SNSE-intro}, for both periodic and Dirichlet boundary conditions,   has a unique invariant measure  $\mu_\eps$ for $\alpha>0$, (to this purpose, see   \cite[Corollary 9.1 and Remark 4.1 (c)]{Brz+Li_2006}). However,  an essential tool in proving the LDP is given by  the  exponential estimates for the invariant measures and we have  been able to prove them only in the case of periodic boundary conditions and  $\alpha >1$ ( see Theorem \ref{thm-exponential}). As a matter of fact, we do not know if, even in the case of periodic boundary conditions, such exponential estiamates are true without assuming that the covariance of the noise is a trace-class operator. This is the reason why, already from the very beginning, we assume that our problem is posed on a 2-D torus and that $\alpha >1$.

Let us conclude this introduction by briefly describing the content of our paper. Section \ref{sec-prel} is devoted to presenting basic notation and preliminaries.  We try to explain the differences and similarities between the NSES with periodic and Dirichlet boundary conditions which lead us to consider only the latter case. In particular, we  prove some estimates concerning the nonlinearity $B$ with respect to norm in different fractional domains of the Stokes operator $\rA$, see Propositions
\ref{prop-Leray-fractional-alpha}  and \ref{prop-Leray-fractional-alpha2}.

In  Section \ref{sec-skeleton} we discuss the  skeleton equation and, in addition to recalling some fundamental and useful results (also from our previous work \cite{BCF_2013}), we also discuss their generalizations to the general  case $\alpha>0$, valid however only for the case of NSEs on a 2-D torus.

In Section \ref{sec-action} we introduce the action functional, for the large deviation principle in $C([0,T];\rH)$ associated with the family of solutions $\{u_\eps\}_{\eps>0}$ of equation \eqref{eqn-SNSE-intro}, and the corresponding quasi-potential. We formulate generalizations of the corresponding results from \cite{BCF_2013}  to the general  case $\alpha>0$, again, valid only for the case of NSEs on a 2-D torus. Moreover, we state our main result, i.e. Theorem \ref{thm-main}, about the LDP for the family of invariant measures $\{\mu_\eps\}_{\eps>0}$ for the stochastic NSEs on a 2-D torus.
The remainder of the paper is devoted to the proof of that result.

So, in
Section \ref{sec-exponentail} we formulate and prove Theorem \ref{thm-exponential} about exponential estimates for the family of probability measures $\big\{\mu_\eps \big\}_{\eps >0}$. This result is based on  the uniqueness and ergodicity of each invariant measure $\mu_\eps$. The basic ingredient in this proof is also Lemma \ref{lem-exponential}, about uniform exponential estimates for the solutions $u_\eps$ of equation \eqref{eqn-SNSE-intro}. Our proof is a simplification (and clarification) of a proof of a more general result from \cite{Goldys+Maslowski_2005}. However, we should note that another proof of such a result is possible, which is based on an earlier paper \cite{Brz+Peszat_2000} by the first named authour and Peszat, see Appendix \ref{app-exponential estimates}. In both Section \ref{sec-exponentail} and Appendix \ref{app-exponential estimates} we have made some comments about the  corresponding results for stochastic NSEs with multiplicative noise. However, a study of the corresponding LDP is postponed till another publication.

Let us continue with the description of the content of our paper. In Section \ref{sec-lower} we continue with the proof of Theorem \ref{thm-main} and show that the invariant measures $\mu_\eps$ satisfy an appropriate lower bounds, see Theorem \ref{thm-lower}. In inequality \eqref{cb60} we already see the relationship between the invariant measures $\mu_\eps$ and the quasi-potential $\rU$.

Sections  \ref{sec-upper1} and \ref{sec-upper2} are devoted to the formulation and proof of appropriate lower bounds satisfied by the invariant measures $\mu_\eps$, see Theorem \ref{thm-upper}.

The paper is concluded with two appendices. The first one of them, as already mentioned, is devoted to an alternative proof of Lemma \ref{lem-exponential}. The second one is devoted to a proof of precise behavior for large negative time  of solutions to the skeleton equation \eqref{eqn-SNSE-control} on the negative half-line $(-\infty,0]$.

\subsection*{Acknowledgments}
The first named author would like to thanks Department of Mathematics,
University of Maryland for it's hospitality during his visit in September 2013 during which this project was initiated.  Research of the second named author was partly supported by NSF grant DMS-1407615. Talks on   preliminary versions of the results from this paper were given by the first named authour at workshops at Lyon, Krak{\'o}w and Loughborough. After, one of them we were informed by A. Shirkyan about a paper \cite{Martirosyan_2015} by D. Martirosyan who studies LDP for invariant measure for stochastic wave equations.

\bigskip

\section{Notation and preliminaries}\label{sec-prel}

Our main results are formulated for the stochastic Navier-Stokes equations with periodic boundary conditions. Hence we
begin with a brief introduction to the relevant notation in this case; all the mathematical background can
be found in the small book \cite{Temam_1983} by Temam.  Here we will not recall the  notation  in the case of  the Dirichlet boundary conditions but only refer the reader to our earlier
paper \cite{BCF_2013}. Some of our results are true also in this case. Proper generalization to  this case, as well to the case of multiplicative noise, will be a subject of a forthcoming publication.

We denote here by  $\mathbb{T}^2$ the two dimensional torus of fixed dimensions $L\times L$. The space $\rH$ is equal to
\[\rH=\{ u\in L_0^2(\mathbb{T}^2,\mathbb{R}^2): \divv (u)=0 \mbox{ and } \gamma_\nu(u)_{\vert \Gamma_{j+2}}=-\gamma_\nu(u)_{\vert \Gamma_{j}}, \; j=1,2\}, \]
where $L_0^2(\mathbb{T}^2,\mathbb{R}^2)$ is the Hilbert space consisting of those $u\in L^2(\mathbb{T}^2,\mathbb{R}^2)$ which satisfy the condition
\begin{equation}\label{eqn-int=0}
\int_{\mathbb{T}^2} u(x)\, dx=0,
\end{equation}
$\gamma_\nu$ is the bounded linear map defined on divergence free vectors in $L^2(\mathbb{T}^2,\mathbb{R}^2)$ with values in the dual space of $H^{\frac 12}(\partial \mathbb{T}^2)$ (the image in $L^2(\partial \mathbb{T}^2)$ of the trace operator $H^1( \mathbb{T}^2)\to L^2(\partial \mathbb{T}^2)$), such that $\gamma_\nu(u)$ coincides with the restriction of $u\cdot \nu$ to $\partial \mathbb{T}^2$, if $u \in\,\overline{D(\mathbb{T}^2)}$, and
$\Gamma_j$, $j=1,\cdots,4$ are the four (not disjoint) parts of the boundary of $\partial(\mathbb{T}^2)$ defined by, for $ j=1,2$,
\[
\Gamma_j=\{ x=(x_1,x_2) \in [0,L]^2: x_j=0\},\;\Gamma_{j+2}=\{ x=(x_1,x_2) \in [0,L]^2: x_j=L\}.
\]

We also define the vorticity space $\rV$ by setting
\begin{equation}\label{eqn-V space}
\left.
\begin{array}{rcl}
\rV&=&\Bigl\{ u\in \rH : D_ju \in L^2(\mathbb{T}^2,\mathbb{R}^2),\ u_{\vert \Gamma_{j+2}}\circ \tau_j=u_{\vert \Gamma_{j}}, \; j=1,2\Bigr\},
\end{array}
\right.
\end{equation}
 where $D_j$, $j=1,2$, are  the 1st order weak derivatives in the interior of the torus.

Because of  condition \eqref{eqn-int=0},   the norm on the space $\rV$ induced by the norm from the Sobolev spaces $\mathbb{H}^{1,2}$
is equivalent to the following one
\[
\big(u,v\big)_\rV=\sum_{i,j=1}^{2}\int_{\mathcal{O}}{\frac{\partial
u_{j}}{\partial x_{i}}\frac{\partial v_{j}}{\partial
x_{i}}}\,{d}x,\; u,v\in\rV.
\]

The Stokes operator $\rA$ can be defined in a natural way as
\begin{equation}
\label{def-A} \left\{
\begin{array}{ll}
D(\mathrm{A}) &= \rV \cap H^{2,2}(\mathbb{T}^2,\mathbb{R}^2)
\\
&\vspace{.1cm} \\
\mathrm{A}u&=-P \Delta , \, u \in
D(\mathrm{A}),
\end{array}
\right.
\end{equation}
where \[
\mathrm{P}:\mathbb{L}^2(\mathcal{O}) \rightarrow \mathrm{H}\]
is  the orthogonal projection,  called usually the Leray-Helmholtz projection.

It is well known that $A$ is a self-adjoint positive operator in $\rH$. In fact, its eigenvectors and eigenvalue can be explicitly found.
In particular, $A$ has bounded imaginary powers and thus by \cite[Remark 2 in 1.15.2]{Triebel_1978}, the domains of the fractional powers of $A$ are equal (with equivalent norms) to the complex interpolation spaces between $D(\rA)$ and $\rH$, i.e.
\begin{equation}
\label{eqn-fractional}
D(\mathrm{A}^\theta) = [\rH, D(\mathrm{A})]_\theta, \;\; \theta \in (0,1).
\end{equation}
This means that
\begin{eqnarray}\label{eqn-fractional domains}
D(\mathrm{A}^\theta)&=& \rH  \cap H^{2\theta,2}(\mathbb{T}^2,\mathbb{R}^2).
\end{eqnarray}
Moreover, it is well known,  see for instance    \cite[p. 57]{Temam_1997}, that $\mathrm{V}=D(\mathrm{A}^{1/2})$.

It follows from the above that, contrary to the Dirichlet boundary case, for every $\alpha\geq 0$ the Leray-Helmholtz projection
\begin{equation}\label{eqn-Leray-Helmholtz}
 P: H^{\alpha,2}(\mathbb{T}^2,\mathbb{R}^2) \to D(\rA^\alpha)
\end{equation}
  is a bounded linear map. To this purpose, compare with    \cite[Proposition 2.1]{BCF_2013} in the bounded domain case.

  The Stokes operator $\rA$ satisfies all the properties known in the bounded domain case, inclusive the strict positivity property
  \[
  \lb \rA u, u\rb_{\rH} \geq \lambda_1 \vert u \vert_{\rH},
  \]
with $\lambda_1=\frac{4\pi^2}{L^2}$.

Now, consider the trilinear form $b$ on $V\times V\times V$ given by
\[
b(u,v,w)=\sum_{i,j=1}^{2}\int_{\mathcal{O}}u_{i}{\frac{\partial v_{j}%
}{\partial x_{i}}}w_{j}\,\,{d}x,\quad u,v,w\in \rV.
\]
It is known that  $b$ is a continuous trilinear form such that
\begin{equation}
\label{eqn:b01}
b(u,v,w)=-b(u,w,v),
\quad  \, u\in \mathrm{V}, v, w\in
\mathbb{H}_0^{1}(\mathcal{O}),
\end{equation}
and, for some constant $c>0$ (see  for instance \cite[Lemma 1.3, p.163]{Temam_2001}  and  \cite{Temam_1997}),

\begin{equation}
\begin{aligned}
\label{eqn:4.0a}
 \vert b(u,v,w)  \vert \leq c\left\{
\begin{array}{ll}
 \vert u\vert_\rH ^{1/2}\vert \nabla u\vert_\rH^{1/2}\vert \nabla  v\vert _\rH^{1/2}\vert \mathrm{A}v\vert_\rH ^{1/2}\vert w\vert_\rH  &
\quad \, u \in \mathrm{V}, v\in D(\mathrm{A}), w\in \mathrm{H}\\
\vert u\vert_\rH ^{1/2}\vert \mathrm{A}u\vert_\rH ^{1/2}\vert \nabla  v\vert _\rH\vert w\vert_\rH & \quad
 u \in D(\mathrm{A}), v\in \mathrm{V}, w\in \mathrm{H}\\
\vert u\vert_\rH \vert \nabla  v\vert_\rH \vert w\vert_\rH ^{1/2}\vert \mathrm{A}w\vert _\rH^{1/2} & \quad
 u \in \mathrm{H}, v\in \mathrm{V}, w\in D(\mathrm{A})\\
\vert u\vert _\rH^{1/2}\vert \nabla  u\vert_\rH ^{1/2}\vert \nabla  v\vert_\rH \vert w\vert_\rH ^{1/2}\vert \nabla  w\vert_\rH ^{1/2} & \quad
 u, v, w \in \mathrm{V}.
\end{array}
\right.
\end{aligned}
\end{equation}

Next, define  the bilinear map $B:\rV\times \rV\rightarrow \rV^{\prime}$, by setting
\[
\left\langle B(u,v),w\right\rangle =b(u,v,w),\quad u,v,w\in \rV,
\]
and the homogenous polynomial of second degree $B:\rV \rightarrow \rV^{\prime}$ by

\[
B(u)=B(u,u),\; u\in \rV.
\]
From the first inequality in \eqref{eqn:4.0a}, we have that
if $v \in\,D(\rA)$, then $B(u,v) \in\,H$ and the following inequality  follows directly \begin{equation}
\label{ineq-B01}
\vert \rB(u,v) \vert_\rH^2 \leq C  \vert u\vert_\rH\vert \nabla u\vert_\rH\vert \nabla  v\vert_\rH\vert \mathrm{A}v\vert_\rH, \; u\in \rV, \, v\in D(\rA).
\end{equation}
Moreover, the following identity is a direct consequence of \eqref{eqn:b01}.
\begin{equation}
\label{eqn:b02}
\lb  \rB(u,v),v \rb =0,\;\;  u,v\in \rV.
\end{equation}

Furthermore,  we have  the following property involving the nonlinear term $B$ and the Stokes operator $\rA$
 \begin{eqnarray}\label{eqn-A-B}
\left<\rA u,B(u,u)\right>_\rH=0, \;\; u\in D(\rA),
\end{eqnarray}
see \cite[Lemma 3.1]{Temam_1983} for a proof.

Let us also recall the following facts (see \cite[Lemma 4.2]{Brz+Li_2006} and \cite{Temam_2001}).

\begin{lemma}\label{lem:form-b} The trilinear map
$b:\mathrm{V}\times \mathrm{V} \times \mathrm{V} \to \mathbb{R}$
has a unique extension to a bounded trilinear map from
$\mathbb{L}^4(\mathcal{O}) \times (\mathbb{L}^4(\mathcal{O})
\cap\mathrm{H})\times \mathrm{V}$ and from
$\mathbb{L}^4(\mathcal{O}) \times \mathrm{V}\times
\mathbb{L}^4(\mathcal{O})$ into $\mathbb{R}$. Moreover, $B$ maps
$\mathbb{L}^4(\mathcal{O})\cap\mathrm{H}$ (and so $\mathrm{V}$)
into $\mathrm{V}^\prime$ and
\begin{equation}\label{eqn:4.0}
\vert B(u) \vert_{\mathrm{V}^\prime} \leq C_1\vert u
\vert^2_{\mathbb{L}^4(\mathcal{O})} \leq 2^{1/2}C_1 \vert u \vert_\rH  \vert
\nabla  u \vert _{\mathbb{L}^2(\mathcal{O})}\leq C_2 \vert u \vert_\rH \vert u \vert_\rV \leq  C_3 \vert u \vert_\rV^2 , \quad
 u \in \mathrm{V}.
\end{equation}
\end{lemma}

\begin{lemma}\label{lem-B}
For any $T\in  (0,\infty]$ and  for any $u\in L^2(0,T;D(\rA))$ with $u^\prime \in  L^2(0,T;\rH)$, we have
\[\int_0^T \vert \rB(u(t),u(t)) \vert_\rH^2 \, dt <\infty.\]
\end{lemma}

The restriction of the map $\rB$ to the space $D(\rA)\times D(\rA)$ has also the following representation
\begin{equation}
\label{eqn-B-using-LH}
  \rB(u,v)= P( u\nabla v)=P(\sum_{j=1}^2 u^jD_jv), \;\; u,v\in D(\rA).
\end{equation}
In view of \eqref{eqn-Leray-Helmholtz},  the above
representation  allows us to prove the following property of the map $B$ (compare with a weaker result  in \cite[Proposition 2.5]{BCF_2013} for the Dirichlet boundary case).

\begin{proposition}\label{prop-Leray-fractional-alpha} Assume that $\alpha \in (0,1] $. Then for any $s \in\,(1,2]$ there exists a constant $c>0$ such that
\begin{equation}
\label{ineqn-B-fractional}
  \vert \rB(u,v)\vert_{D(\rA^{\frac\alpha2})} \leq  c \vert u\vert_{D(\rA^{\frac{s}2})}  \vert v\vert_{D(\rA^{\frac{\alpha+1}2})} , \;\; u,v\in D(\rA).
\end{equation}
\end{proposition}

\begin{proof}
In view of equality \eqref{eqn-B-using-LH}, since by \eqref{eqn-Leray-Helmholtz}
 the  Leray-Helmholtz projection $P$  is a well defined and continuous map  from
$\mathbb{H}^{\alpha}(\mathcal{O})$ into  $ D(\rA^{\frac\alpha2})$ and since the norms in the spaces $D(\rA^{\frac{s}2})$ are equivalent to norms in $\mathbb{H}^s(\mathcal{O})$, it is enough to show that
\[
\vert u\nabla v \vert_{\mathbb{H}^{\alpha}} \leq  c \vert u\vert_{\mathbb{H}^{s}}  \vert v\vert_{\mathbb{H}^{\alpha+1}} , \;\; u,v\in \mathbb{H}^2(\mathcal{O}).
\]
Thus, it is sufficient to prove for scalar valued functions
\begin{equation}\label{ineq-001}
\vert fg \vert_{H^{\alpha}} \leq  c \vert f\vert_{H^{s}}  \vert g\vert_{H^{\alpha}} , \;\; f,g\in H^2.
\end{equation}

First we consider the case $\alpha=0$. In this case it is sufficient to assume that $s \in (1,2)$ and we have
\begin{eqnarray*}
\vert fg \vert_{L^2}
&\leq&  \vert  f \vert_{L^\infty} \vert g \vert_{L^2}  \leq
\vert  f \vert_{H^s}\vert g \vert_{L^2}
\end{eqnarray*}
by the  Gagliado-Nirenberg inequality, which implies that $H^s \embed L^\infty$ continuously.

 Secondly, we consider the case $\alpha=1$. Also in this case it is sufficient to assume that $s \in (1,2)$. Then, by the Sobolev
Gagliado-Nirenberg inequalities, we have
\begin{eqnarray*}
\vert \nabla (fg) \vert_{L^2}   &\leq &  \vert g \nabla f \vert_{L^2} + \vert f \nabla g \vert_{L^2}\\
&\leq&  \vert \nabla f \vert_{L^p}\vert g \vert_{L^q} + \vert f \vert_{L^\infty} \vert \nabla g \vert_{L^2}
\\
&\leq&  \vert \nabla f \vert_{L^p}\vert g \vert_{H^1} + \vert f \vert_{H^s} \vert \nabla g \vert_{L^2}\\
&\leq&  \vert  f \vert_{H^s}\vert g \vert_{H^1} + \vert f \vert_{H^s} \vert \nabla g \vert_{L^2}
\end{eqnarray*}
where $p,q \in (2,\infty)$ are such that $\frac1p+\frac1q=\frac12$ and $\frac12=\frac1p+\frac{s-1}2$, i.e. $\frac1p=1-\frac{s}2$. From the above two inequalities we trivially deduce that
\begin{eqnarray*}
\vert fg \vert_{H^1}   &\leq &   \sqrt{6} \vert  f \vert_{H^s}\vert g \vert_{H^1}
\end{eqnarray*}
what proves inequality \eqref{ineq-001} for $\alpha=1$.

Finally, let us consider the case $\alpha \in (0,1)$. By  a complex interpolation  argument and the Marcinkiewicz Interpolation Theorem,  we infer that for any $\alpha \in\,(0,1)$
\begin{eqnarray*}
\vert fg \vert_{H^\alpha} &\leq& 6^{\frac\alpha2} \vert  f \vert_{H^s}\vert g \vert_{H^\alpha},
\end{eqnarray*}
so that the proof is complete.
\end{proof}
Since the Sobolev space $H^\alpha$ is an algebra for $\alpha>1$, we have the following result.

\begin{proposition}\label{prop-Leray-fractional-alpha2} Assume that $\alpha \in (1,\infty) $. Then  there exists a constant $c>0$ such that
\begin{equation}
\label{ineqn-B-fractional2}
  \vert \rB(u,v)\vert_{D(\rA^{\frac\alpha2})} \leq  c \vert u\vert_{D(\rA^{\frac{\alpha}2})}  \vert v\vert_{D(\rA^{\frac{\alpha+1}2})} , \;\; u \in\,D(\rA^{\frac \alpha 2}),\ ,v \in\, D(\rA^{\frac{1+\alpha}2}).
\end{equation}
\end{proposition}

\begin{proof}
Let us fix $\alpha \in (1,\infty) $.
In view of equality \eqref{eqn-B-using-LH}, as in the proof of Proposition \ref{prop-Leray-fractional-alpha} it is enough to show that
\[
\vert u\nabla v \vert_{\mathbb{H}^{\alpha}} \leq  c \vert u\vert_{\mathbb{H}^{\alpha}}  \vert v\vert_{\mathbb{H}^{\alpha+1}} , \;\; u \in\, \mathbb{H}^\alpha(\mathcal{O}),\ v \in\, \mathbb{H}^{\alpha+1}(\mathcal{O}).
\]
Since the Sobolev space $H^\alpha$ is an algebra and $\vert \nabla v \vert_{\mathbb{H}^{\alpha}} \leq c \vert  v \vert_{\mathbb{H}^{\alpha+1}}$, the result follows.

\end{proof}

\begin{remark}\label{rem-form B} One consequence of  Propositions
\ref{prop-Leray-fractional-alpha}  and \ref{prop-Leray-fractional-alpha2} is that the 2-D NSEs with periodic boundary conditions are locally well posed in the space $D(\rA^{\frac{\beta}2})$, for every $\beta \geq 0$. To be precise for every $u_0\in D(\rA^{\frac{\beta}2})$ and every $f \in L^2_{{\rm loc}}([0,\infty);D(\rA^{\frac{\beta}2-\frac12})$ there exists $T>0$ and a strong solution $u$ defined on $[0,T]$. This local existence result  is  well known for $\beta \in \{0,1\}$ for 2-D NSEs with either Dirichlet or periodic boundary conditions. Moreover, it is rather a folk result for
$\beta \in (0,1)$. However,  in \cite{BCF_2013} we proved it  to be true also  for $\beta \in (1,\frac 32)$. The difference between the NSEs with general boundary conditions and NSEs on a torus stems from the fact that, while Propositions
\ref{prop-Leray-fractional-alpha}  and \ref{prop-Leray-fractional-alpha2} hold in the latter case for any $\alpha>0$, we have  been able to establish a corresponding result in the former case  for only $\alpha \in [0,\frac12)$. And the root for this difference lies in the properties of the the Leray-Helmholtz projection $P$. Actually, while in the latter case, it is a bounded linear map from
$ H^{\alpha,2}(\mathbb{T}^2,\mathbb{R}^2)$ into $D(\rA^\alpha)$, in the former case we have been able  to prove an analogous result only for $\a \in\,(0,\frac 12)$.
As in \cite{BCF_2013}, by using the global well-posedness in $\rH$, local well-posedness in $D(\rA^{\frac{\beta}2})$ implies also  global well-posedness. See Proposition \ref{prop-NSEs-global} for precise formulations of these results.
 \end{remark}

\begin{remark}
Similar inequalities to those in Propositions \ref{prop-Leray-fractional-alpha} and \ref{prop-Leray-fractional-alpha2} have also been studied in \cite{Brz+Dgarival_2015}.
\end{remark}
\bigskip

\section{The skeleton equation}\label{sec-skeleton}

We consider here the following Navier-Stokes equation
\begin{equation}
\label{eqn_NSE01}
\left\{
\begin{array}{l}
u^\prime(t)+ \rA u(t)+\rB(u(t),u(t))=f(t), \ \ \ \ t\in (a,b),\\
\vspace{.1cm} \\
u(a)=u_0,
\end{array}\right.\end{equation}
where $-\infty<a<b<\infty$.

\begin{definition}
Given any $f\in L^2(a,b;\rV^\prime)$ and $u_0\in \rH$, a solution to problem \eqref{eqn_NSE01} is a function $u\in L^2(a,b;\rV)$ such that $u^\prime \in L^2(a,b;\rV^\prime)$,  $u(a)=u_0$ \footnote{It is known (see for instance \cite[Lemma III.1.2]{Temam_2001}) that these two properties of $u$ imply that $u$ is almost everywhere equal to a function $\bar{u}\in C([a,b],\rH)$. Thus, when we later write $u(a)$  we mean $\bar{u}(0)$.} and \eqref{eqn_NSE01} is fulfilled.
\end{definition}

As shown in   \cite[Theorems III.3.1/2]{Temam_2001}, for every  $f\in L^2(a,b;\rV^\prime)$ and $u_0\in \rH$ there exists exactly one solution $u$ to  problem \eqref{eqn_NSE01}.

\begin{lemma}
For any $r>0$, there exists $c_r>0$ such that, if  $u\in C([T,+\infty);\rH)$, for some $T \in\,\mathbb{R}$, and
$$u^\prime+\rA u+\rB(u,u)=:f \in L^2(T,+\infty;\rH),$$
then
\begin{equation}
\label{cb3}
|f|_{L^2(T,+\infty;\rH)}\leq r\Longrightarrow |u|_{C([T,+\infty);\rH)}\leq c_r+|u(T)|_{\rH}.\end{equation}
Moreover
\begin{equation}
\label{cb24}
 |u(t)|_{\rH}\leq \ex{-\lambda_1 (t-T)}|u(T)|_{\rH}+\frac 1{\lambda_1}|f|^2_{L^2(T,+\infty;\rH)},\ \ \ \ t \geq T.\end{equation}
\end{lemma}

\begin{proof}
We have
\begin{equation}
\label{cb23}
\frac 12\frac{d}{dt}|u(t)|^2_{\rH}+|u(t)|^2_{\rV}\leq \frac 12|u(t)|_{\rV}^2+\frac 1{2\lambda_1}|f(t)|_{\rH}^2.\end{equation}
This implies that
\begin{equation}
\label{cb4}
|u|^2_{C([a,b];\rH)}+|u|^2_{L^2(a,b;\rV)}\leq |u(a)|^2_{\rH}+\frac 1{\lambda_1}|f|^2_{L^2(a,b;\rH)},\end{equation}
which yields \eqref{cb3}, if we take $a=T$ and $b=+\infty$.

Moreover, from \eqref{cb23} we have
\[\frac 12\frac{d}{dt}|u(t)|^2_{\rH}+\lambda_1\,|u(t)|^2_{\rH}\leq\frac 1{\lambda_1}|f(t)|_{\rH}^2,\]
so that, by the Gronwall lemma, for any $a\leq t\leq b\leq +\infty$ we have
\[|u(t)|^2_{\rH}\leq |u(a)|^2_{\rH}\ex{-\lambda_1(t-a)}+\frac 1{\lambda_1}\int_a^t \ex{-\lambda_1(t-s)}|f(s)|_{\rH}^2\,ds,\]
which implies \eqref{cb24}.

\end{proof}

In \cite[Theorem III.3.10]{Temam_2001} it is proven that, if  $f\in L^2(a,b;\rH)$,  the solution $u$ of equation \eqref{eqn_NSE01} has the following properties
$$\sqrt{(\cdot-a)}\,u \in L^2(a,b;D(\rA)) \cap  L^\infty(a,b;\rV),\;\; \sqrt{(\cdot-a)}\,u^\prime \in L^2(a,b;\rH).$$
Moreover,  there exists $c>0$ such that for all $a<b$
\begin{equation}
\label{brc82}
\begin{array}{l}
|\sqrt{(\cdot-a)}\,u|^2_{L^\infty(a,b;\rV)}+|\sqrt{(\cdot-a)}\,u|^2_{L^2(a,b;D(\rA))}\\
\vspace{.1cm}\\
\leq c\, \exp\,\left[c\big( \vert u_0\vert_{\rH}^4+ \vert f\vert^4_{L^2(a,b;\rV^\prime)} \big)\right]\Big( \vert u_0\vert_{\rH}^2+ \vert f\vert^2_{L^2(a,b;\rV^\prime)}+|b-a|\vert f\vert^2_{L^2(a,b;\rH)} \Big).
\end{array}
\end{equation}

In \cite[Proposition 3.3]{BCF_2013} we have also proved the following result
for $\alpha \in (0,\frac12)$ for 2-D NSEs with both Dirichlet and periodic boundary conditions. It turns out that in the latter case it is  true  for any $\alpha\geq 0$.
\begin{proposition}\label{prop-NSEs-global} Suppose that $\alpha \geq 0$.
If $f\in L^2(a,b;D(\rA^{\frac{\alpha}{2}})$ and $u_0\in D(\rA^{\frac{\alpha+1}{2}})$, for some $\a\geq 0$,  then the unique solution $u$  to the problem \eqref{eqn_NSE01} satisfies
\begin{equation}\label{eqn-u-better regularity}
u\in L^2(a,b;D(\rA^{1+\frac{\alpha}{2}}) \cap C([a,b];D(\rA^{\frac{\alpha+1}{2}})),\;\; u^\prime(\cdot)\in L^2(a,b;D(\rA^{\frac{\alpha}{2}})).
\end{equation}
\end{proposition}
\begin{proof}
As discussed in Remark \ref{rem-form B}, the above result follows from Propositions
\ref{prop-Leray-fractional-alpha}  and \ref{prop-Leray-fractional-alpha2}.
The proof of the above result can be accomplished by following line by line  the proof of \cite[Proposition 3.3]{BCF_2013}, which worked for both types of boundary conditions but only for $\alpha \in (0,\frac12)$.
\end{proof}

\bigskip

Now, for any $-\infty\leq a<b\leq \infty$ and for any two reflexive Banach spaces $X$ and $Y$, such that $X \embed Y$ continuously, we denote by $W^{1,2}(a,b;X,Y)$ the space of all $u\in L^2(a,b;X)$ which are weakly differentiable as  $Y$-valued functions and their weak derivative
belongs to $L^2(a,b;Y)$. The space $W^{1,2}(a,b;X,Y)$  is a separable Banach space (and Hilbert if both $X$ and $Y$ are Hilbert spaces), endowed with the natural norm
$$\vert u\vert_{W^{1,2}(a,b;X,Y)}^2= \vert u\vert^2_{ L^2(a,b;X)}+\vert u^\prime\vert^2_{ L^2(a,b;Y)},\;\; u\in W^{1,2}(a,b;X,Y).
$$
Later on, when no ambiguity is possible, we will use  the shortcut notation
\[W^{1,2}(a,b)=W^{1,2}(a,b;D(\rA),\rH).\]

\medskip

\begin{definition}
\label{def-very weak}
Assume that $-\infty \leq a < b\leq \infty$ and $f\in L^2_{\textrm{loc}}(a,b;\rH)$. A function $u\in C((a,b);\rH)$ is called a {\em very weak solution} to the Navier-Stokes equations \eqref{eqn_NSE01} on the interval $(a,b)$ if for all $\phi \in C^\infty ((a,b)\times D)$, such that $\text{div} \phi =0$ on $(a,b)\times  D$,
\begin{equation}
\label{eqn-very weak}
\begin{array}{l}
\ds{\int_D u(t_1,\xi) \phi(t_1,\xi) d\xi= \int_D u(t_0,\xi) \phi(t_0,\xi) d\xi}\\
\vs\\
\ds{+\int_{[t_0,t_1]\times D} u(s,\xi) (\partial_s\phi(s,\xi)
+\Delta \phi(s,\xi))\, ds d\xi} \\
\vs\\
\ds{+ \int_{t_0}^{t_1} b(u(s),u(s),\phi(s))\, ds + \int_{[t_0,t_1]\times D}  f(s,\xi) \cdot \phi(s,\xi) \, ds d\xi,}
\end{array}
\end{equation}
 for all $a<t_0<t_1<b$.

\end{definition}

By adapting some  of the results from \cite{Lions+M_2001} to the $2$-dimensional case, it is possible   to prove the following result.

\begin{proposition}\label{prop-very weak}
Assume that $-\infty \leq a < b\leq \infty$ and $f\in L^2_{\textrm{loc}}((a,b);\rH)$. Suppose that the functions $u, v\ \in\, C((a,b);\rH)$ are very weak solutions to the Navier-Stokes equations \eqref{eqn_NSE01} on the interval $(a,b)$, with   $u(t_0)=v(t_0)$,  for some $t_0\in (a,b)$. Then $u(t)=v(t)$ for all $t\geq t_0$.
\end{proposition}

\medskip

\begin{definition}
\label{def-action}
Assume that $-\infty \leq a < b\leq \infty$. Given a  function $u\in C((a,b);\rH)$ we say that
\[u^\prime+ \rA u +\rB(u,u)\ \in\, L^2(a,b;\rH),\ \ \  (\text{resp.}\  \in\, L^2_{\textrm{loc}}((a,b);\rH))\]
 if there exists $f\in L^2(a,b;\rH)$, (resp. $f \in L^2_{\textrm{loc}}((a,b);\rH)$) such that $u$ is a very weak solution of the Navier-Stokes equations \eqref{eqn_NSE01} on the interval $(a,b)$.

Obviously, the corresponding function $f$ is unique and we will denote it by $\mathcal{H}(u)$, i.e.
\begin{eqnarray}
\label{eqn-A3}
[\cH(u)](t)&:=&u^\prime(t)+\rA u(t)+\rB(u(t),u(t)), \;\; t \in (a,b).
\end{eqnarray}
\end{definition}

In \cite[Proposition 10.2]{BCF_2013} we have proved the following result.
\begin{lemma}
\label{lemma4.2} Assume that  $\alpha \in (0,\frac12)$.
Assume that $u \in\,C((-\infty,0];\rH)$ is such that
${\mathcal H}(u):=u^\prime+\rA u+\rB(u,u) \in\,L^2(-\infty,0;D(\rA^{\frac \a 2}))$ and there exists $\{t_n\}\downarrow -\infty$, such that
\[\lim_{n\to \infty}|u(t_n)|_{\rH}=0.\]
Then $u \in\,W^{1,2}(-\infty,0;D(\rA^{1+\frac{\a}2}),D(\rA^{\frac{\a}2}))$, $u(0) \in\,D(\rA^{\frac{\a+1}2})$ and
\[\lim_{t\to-\infty}|u(t)|_{D(\rA^{\frac{\a}2+\frac12})}=0.\]

\end{lemma}
In Appendix \ref{section-negative}, we generalize the above result, again only in the case of NSEs on a torus, to the case of any $\alpha >0$.

\bigskip

\section{LDP for stochastic NSEs on a 2-D torus}\label{sec-action}

For any fixed $\eps \in (0,1]$ and $x \in\,\rH$, we consider the problem
\begin{equation}
\label{eqn-SNSE-eps}
du(t)+\rA u(t)+\rB(u(t),u(t))=\sqrt{\eps}\,dw(t),\ \ \ \ u(0)=x.
\end{equation}
Here $w=\big\{w(t)\big\}_{t\geq 0}$ is an $\rH$-valued  Wiener process with  reproducing kernel Hilbert space denoted by $\rK$. In particular $\rK \subset \rH$ and the natural embedding  $\rK \embed \rH$ is $\gamma$-radonifying (i.e. in this case Hilbert-Schmidt). Let us fix a complete orthonormal system  $\{f_k\}_{k \in\,\mathbb{N}}$ of $\rK$ and a sequence $\big\{\beta_k\big)\}_{k=1}^\infty $ of independent Brownian motions defined on some filtered probability space $\big(\Omega, \mathcal{F}, \mathbb{F},\mathbb{P}\big)$, where $\mathbb{F}=\big\{\mathcal{F}_t\big\}_{t\geq 0}$, such that the Wiener process $w$ has the following representation
\begin{equation}
\label{eqn-W}
w(t)=\sum_{k=1}^\infty   \beta_k(t)f_k,\ \ \ \ t\geq 0.
\end{equation}
Together with the Wiener process $w$, one can associate a covariance operator $C$, usually denoted by $Q$,  defined by
\[
\lb Ch_1,h_2\rb = \mathbb{E} \big[ \lb h_1, w(1)\rb_\rH \, \lb w(1),h_2 \rb_\rH \big], \;\; h_1,h_2 \in \rH.
\]
It is well known that $C$ is a self-adjoint and trace class operator in $\rH$, see for instance the monograph \cite{DaPrato+Z_1992}. If $i:\rK \embed \rH$ is the natural embedding, then $C=i i^\ast$ and $\rK=R(C^{\frac12})$, see for instance \cite{Brz+vN_1999}.
In this paper we consider very special type of Wiener process for which the  operator $Q:=C^{\frac12}$ is an isomorphism of $\rH$ onto $D(\rA^{\frac{\alpha}{2}})$ for some $\alpha>1$. In other words, when $\rK=D(\rA^{\frac{\alpha}{2}})$, for some $\alpha>1$.

Since $C$ is trace class, we infer that
 \[\sum_{k=1}^\infty |Qf_k|_{\rH}^2=\text{Tr}\,[Q Q^\star]<\infty.\]

\delc{ given by
\[W(t)=\sum_{k=1}^\infty Q e_k \beta_k(t),\ \ \ \ t\geq 0,\]
where $\{e_k\}_{k \in\,\mathbb{N}}$ is the basis which diagonalizes the operator $\rA$, $\{\beta_k(t)\}_{k \in \mathbb{N}}$ is a sequence of independent Brownian motions, all defined on the same stochastic basis $(\Omega, \mathcal {F}, \mathcal {F}_t, \mathbb{P})$, and $Q$ is a positive bounded linear operator on $\rH$. We assume here that %
 $Q: \rH \to D(\rA^{\frac{\alpha}{2}})$  is an isomorhism, for some $\alpha>1$. In particular, as $\a_k\sim k$, we have that $Q \in\,{\mathcal L}_2(\rH)$, that is, for any complete orthonormal system $\{f_k\}_{k \in\,\mathbb{N}}$ in $\rH$
 \[\sum_{k=1}^\infty |Qf_k|_{\rH}^2=\text{Tr}\,[Q Q^\star]<\infty.\]
 This implies that $W(t) \in\,L^2(\Omega;\rH)$, $t\geq 0$.
}

\medskip

As proven e.g. in \cite{Flandoli_1994}, for any $T>0$, $\eps \in\,(0,1]$ and $x \in\,\rH$, equation \eqref{eqn-SNSE-eps} admits a
unique solution $u^x_{\eps} \in L^p(\Omega;C([0,T];\rH))$, for any $p\geq 1$. Moreover, there exists an  invariant measure
$\nu_\eps$ for the Markov process generated by equation \eqref{eqn-SNSE-eps}, i.e. a Borel probability measure on $\rH$ such that
for every bounded a continuous function $\varphi:\rH\to\mathbb{R}$ and every $t\geq 0$,
\begin{equation}
\label{eqn-invariant measure}
 \int_{\rH}  \mathbb{E}\,\varphi(u^x_\eps(t)) \,\nu_\eps(dx) = \int_{\rH}\varphi(x)\,\nu_\eps(dx).
\end{equation}

 As it is known, see \cite{Ferrario_1999}, for every $\eps>0$
 this invariant measure $\nu_\eps$ is unique and ergodic. Thus is particular,  for any  bounded Borel measurable   function $\varphi:\rH\to\mathbb{R}$,

\begin{equation}
\label{eqn-ergodicity}
\int_{\rH}\varphi(x)\,\nu_\eps(dx)=\lim_{T\to\infty}\frac 1T\int_0^T \mathbb{E}\,\varphi(u^0_\eps(t))\,dt.
\end{equation}

\medskip

Now, let $u \in\,C([a,b];\rH)$, with $-\infty<a<b<+\infty$.  If
\[\mathcal{H}(u):= u^\prime+\rA u +\rB(u,u)  \in L^2_{\textrm{loc}}((a,b);D(A^{\frac{\alpha}{2}})),\]
 then for any   $a<t_0<t_1<b$ we define
\begin{eqnarray}
\label{eqn-B5}
S_{t_0,t_1}(u)&:=&\frac12 \int_{t_0}^{t_1}\vert Q^{-1} \cH(u)(t)\vert_{\rH}^2\, dt.
\end{eqnarray}
Note that since $Q$ is a bounded operator in $\rH$, we have the following useful inequality
\begin{equation}
\label{ineq-useful}
\int_{t_0}^{t_1}\vert  \cH(u)(t)\vert_{\rH}^2\, dt  \leq 2 \Vert Q\Vert_{\mathcal{L}(\rH,\rH)}^2 S_{t_0,t_1}(u).
\end{equation}

If $\mathcal{H}(u) \notin L^2_{\textrm{loc}}((a,b);D(A^{\frac{\alpha}{2}})$, we define
$S_{t_0,t_1}(u)=+\infty$.
Moreover, for every $T>0$, we denote
\[S_{-T}:=S_{-T,0},\ \ \ \ S_T:=S_{0,T}.\] In particular, when $a=-\infty$ and $b\geq 0$, we set
\begin{eqnarray}
\label{eqn-A5}
S_{-\infty}(u)&:=&\frac12 \int_{-\infty}^0\vert Q^{-1}\cH(u)(t)\vert_{\rH}^2\, dt.
\end{eqnarray}

An obvious sufficient condition for the finiteness of  $S_{t_0,t_1}(u)$ is that $u^\prime$, $Au$  and $B(u,u)$ all belong to $L^2(t_0,t_1;D(A^{\frac{\alpha}{2}})$. In fact, as we  proved in \cite[Lemma 3.9]{BCF_2013}, in the case of 2-D NSEs with both Dirichlet and periodic boundary conditions, when  $\alpha \in (0,\frac12)$,  this is not so far from being a necessary condition. As earlier for Proposition \ref{prop-NSEs-global}, it turns out that in the latter case, \cite[Lemma 3.9]{BCF_2013} holds  true  for any $\alpha\geq 0$.

\begin{lemma}\label{lem-clar-2}
Suppose that $\a\geq 0$ and  $-\infty<a<b<\infty$. \\If a function  $u\in C([a,b];\rH)$ satisfies
\[
u^\prime+\rA u+\rB(u,u)\in L^2(a,b;D(\rA^{\frac{\alpha}2})),\]
then $u(b)\in D(\rA^{\frac{\alpha+1}2})$ and $ u\in W^{1,2}\big(t_0,b;D(\rA^{\frac{\alpha}2+1}),D(\rA^{\frac{\alpha}2})\big)$,
for any $t_0\in (a,b)$.\\
Moreover, if $u(a)\in D(\rA^{\frac{\alpha+1}2})$, then $u\in W^{1,2}\big(a,b;D(\rA^{\frac{\alpha}2+1}),D(\rA^{\frac{\alpha}2})\big)$.
\end{lemma}
\begin{proof}
As discussed in Remark \ref{rem-form B}, the above result follows from Propositions
\ref{prop-Leray-fractional-alpha}  and \ref{prop-Leray-fractional-alpha2}.
The proof of the above result can be accomplished by following the line of  proof of \cite[Lemma 3.9]{BCF_2013} which worked for both types of boundary conditions but only for $\alpha \in (0,\frac12)$.
\end{proof}

\medskip

As  a  consequence of the contraction principle and of some continuity properties of the solution of equation \eqref{eqn-SNSE-eps} proven in \cite{Brz+Li_2006}, in \cite[Theorem 5.3]{BCF_2013} we have shown the following result.
\begin{theorem}
\label{brc.teo.4.2}
For any $x \in\,\rH$, the family
$\{{\mathcal L}(u^x_{\eps})\}_{\eps \in\,(0,1]}$ satisfies the large deviation principle in $C([0,T];\rH)$, with speed $\eps$ and   action functional $S_{T}$, uniformly with respect to $x$ in bounded subsets of $\rH$.\end{theorem}

\medskip

We recall here that a family of probability measures $\{\mu_\eps\}_{\eps>0}$ on some complete metric space $E$ satisfies a large deviation principle, with normalizing constants $\{\beta_\eps\}_{\eps>0}$ such that $\beta(\eps)\downarrow 0$, as $\eps\downarrow 0$, and  action functional $I$, iff  $I$ is a lower-semicontinuous\footnote{This condition is redundant if we also assume condition \eqref{cond-I_r-compact} below.} map from $E$ to $[0,\infty]$ such that
\begin{enumerate}
\item\label{cond-I_r-compact} For each $ r>0$,  the level set
\[I_r:=\left\{x \in\,E\ :\ I(x)\leq r\right\},\ \ \ \ \,\]
is   compact in $E$.
\item {\em Lower bounds:} For every $\bar{x} \in\,E$ and for all  $\delta, \gamma>0$ there exists $\eps_0>0$ such that\footnote{Inequality \eqref{ineq-ld-lb} is trivially satisfied when $I(\bar{x})=\infty$.}
\begin{equation}\label{ineq-ld-lb}\mu_\eps\left(B_E(\bar{x},\delta)\right)\geq \exp\left(-\frac{I(\bar{x})+\gamma}{\beta_\eps}\right),\ \ \ \ \eps\leq \eps_0.
\end{equation}
Here $B_E(\bar{x},\delta)=\{ x\in E: |x-\bar{x}|_{E}<\delta\}$.
\item {\em Upper bounds:} For every $s, \delta$ and $\gamma \in\,(0,s)$ there exists $\eps_0>0$ such that
\begin{equation}\label{ineq-ld-ub}\mu_\eps\left(\{ x\in E: \mbox{dist}_E\left(x, I_s\right)>\delta\}\right)\leq \ex{-\frac{s-\gamma}{\beta_\eps}},\ \ \ \ \eps\leq \eps_0.
\end{equation}

\end{enumerate}
\medskip

Next,  for any $x,y \in\,H$ and $a,b\in \mathbb{R}$, we introduce the following functional spaces
\[\begin{array}{l}
\ds{\mathcal{X}=\big\{ u\in C((-\infty,0];\rH): \; \lim_{t\to-\infty}\vert u(t)\vert_{\rH}=0\big\},}\\
\vs\\
\ds{\mathcal{X}_x=\big\{ u\in \mathcal{X}:u(0)=x\big\},}\\
\vs\\
\ds{C_{x,y}([a,b],\rH)=\{ u \in C([a,b];\rH): u(a)=x,\; u(b)=y\}.}
\end{array}\]
 We endow the space  $\mathcal{X}$ with the topology of uniform convergence on compact intervals, i.e. the topology induced by the metric $\rho$ defined by
$$ \rho(u,v):= \sum_{n=1}^\infty 2^{-n}\left(\sup_{s\in [-n,0]} \vert u(s)-v(s)\vert_{\rH}\wedge 1\right),\;\;\; u,v\in \mathcal{X}.
$$
The set $\mathcal{X}_x$ is a closed in $\mathcal{X}$ and we endow it with the trace topology induced by $\mathcal{X}$.

\medskip

In \cite[Propositions 5.4 and 5.5]{BCF_2013}, we  proved that the functional
\begin{equation}\label{eqn-action functional infty}
S_{-\infty}:\mathcal{X}\to [0,+\infty],
\end{equation}
 is lower-semicontinuous  and has
compact level sets. This result obviously holds for both types of boundary conditions, but only for $\alpha\in [0,\frac12)$. Its proof relied on \cite[Propositions 10.1 and 10.2]{BCF_2013} which we generalize in Appendix \ref{section-negative}. Let us state it for the completeness sake.

\begin{proposition}\label{prop-action-lsc+compact}
Assume that $\alpha \geq 0$.
Then the  functional $S_{-\infty}$ defined by \eqref{eqn-A5}, is  lower-semicontinuous on $\mathcal{X}$. Moreover, its level sets are compact  in $\mathcal{X}$
\end{proposition}

\bigskip

Next, we  define the quasi-potential $\rU$ associated with equation \eqref{eqn-SNSE-eps}, by setting
\begin{equation}
\label{eqn-B4-QP}
\begin{array}{l}
\ds{\rU(x):=\inf\big\{S_{T}(u):  T>0,\, u\in C_{0,x}([0,T];\rH)\big\}}\\
\vspace{.1cm}\\
\ds{=\inf\big\{S_{-T}(u):  T>0,\, u\in C_{0,x}([-T,0];\rH)\big\},\;\;\; x\in \rH.}
\end{array}
\end{equation}

In \cite{BCF_2013}, we have thoroughly studied the functional $\rU$ for the 2-D NSEs, for both Dirichelt and periodic boundary conditions, and we have shown that it satisfies the properties described below for $\alpha \in (0,\frac12)$. The following result generalizes  \cite[Theorem 6.2 and Propositions 6.1, 6.5 and 6.6]{BCF_2013} to  $\alpha \geq 0$ but only, as all our generalizations, for the case of the 2-D NSEs on a torus.
\begin{theorem}
\label{teo4.3}
Assume that $\alpha \geq 0$. Then the following properties hold true.
\begin{enumerate}
\item If $x\in \rH$, then
\begin{equation}
\label{brc23}
\rU(x)<\infty \Longleftrightarrow x\in D(\rA^{\frac{\alpha+1}2}).
\end{equation}
\item If    $x \in \rH$, then
\begin{equation}
\label{eqn-A1}
\rU(x):=\inf\big\{S_{-\infty}(u): u\in \mathcal{X}_x\big\}.
\end{equation}
\item For any $r>0$, the level set
\begin{equation}
\label{eqn-K}
K_r=\{x\in \rH: \rU(x)\leq r\}
\end{equation}
is compact in $\rH$. In particular, the function $\rU: \rH \to [0,\infty]$ is lower semi-continuous.
\item The restriction of the map $\rU$ to the set $D(A^{\frac{\alpha+1}2})$, i.e. the map  \[\rU:D(A^{\frac{1+\beta}2})\to \mathbb{R}\] is continuous.
\end{enumerate}
\end{theorem}
\begin{proof}
The proofs of the above results follow the proofs   of  \cite[Theorem 6.2 and Propositions 6.1, 6.5 and  6.6]{BCF_2013}, while taking into account Propositions
\ref{prop-Leray-fractional-alpha}  and \ref{prop-Leray-fractional-alpha2}.
\end{proof}
\medskip

As mentioned in the introduction, in the present paper we want to prove the following theorem.

\begin{theorem}
\label{thm-main}
The family $\{\nu_\eps\}_{\eps>0}$ of the invariant measures for equation \eqref{eqn-SNSE-eps} satisfies a large deviation principle in $\rH$, with normalizing constants $\beta_\eps=\eps$ and  action functional $\rU$.
\end{theorem}

In Theorem \ref{teo4.3}, we have seen that $\rU$ has compact level sets. Thus, in order to prove Theorem \eqref{thm-main}, in what follows we have only to prove the validity of the lower and upper bounds.

\bigskip

\section{Exponential estimates}\label{sec-exponentail}

In the proof of lower  bounds for the large deviation principle we need to prove that there exists some $\bar{R}>0$ such that
\begin{equation}
\label{cb7}
\lim_{\eps\to 0}\nu_\eps  \big( B_{\rH}^{\rm c}(0,\bar{R})\big)  \delr{|x|_{\rH}\geq \bar{R}}=0.
\end{equation}
On the other hand, in the proof of upper bounds we need something stronger. Actually we need that the convergence to zero in \eqref{cb7} is exponential.

\begin{theorem}
\label{thm-exponential}
For any $s>0$, there exist $\eps_s>0$ and $R_s>0$ such that
\begin{equation}
\label{cb8}
\delr{\nu_\eps\left(|x|_{\rH}\geq R_s\right)}
\nu_\eps \big( B_{\rH}^{\rm c}(0,R_s)\big)
\leq \exp \left(-\frac s\eps\right),\ \ \ \ \eps\leq \eps_s.
\end{equation}

\end{theorem}
This fundamental result will be used in the proof of Theorem \ref{thm-upper}. Let us note that the proof uses the  ergodicity of the invariant measure.

\begin{remark}\label{rem-thm-exponential}

The result from Theorem \ref{thm-exponential} is also true for the stochastic Navier-Stokes equations with multiplicative noise
 \begin{equation}
\label{eqn-SNSE-multiplicative}
du(t)+\rA u(t)+\rB(u(t),u(t))=\sqrt{\eps} g(u)\,dw(t),\ \ \ \ u(0)=u_0.
\end{equation}
where $w(t)$ is a cylindrical Wiener process on some separable Hilbert space $\rK$,  provided  the map
\[
g:\rH \to R(\rK,\rH)
\]
is a continuous and bounded and  there exists    unique ergodic   invariant measure $\nu_\eps$ of the corresponding Markov process.
\end{remark}
\vskip 1 truecm

An essential part of the proof of the above  result is given by the following exponential estimates. Their proof can in fact be traced to the paper \cite{Brz+Peszat_2000}, see also Appendix \ref{app-exponential estimates}, but we present here an independent one based on the use of a suitable Lyapunov function. This proof goes back to the paper \cite{Goldys+Maslowski_2005}, but that paper tried to treat so many cases simultaneously  that we decided to write down an independent  statement and proof.

\begin{lemma}
\label{lem-exponential}
In the framework of Theorem \ref{thm-exponential}, for any arbitrary $\eps>0$ there exists $\bar{\gamma}>0$ such that

\[\mathbb{E}\;\ex{\frac {\bar{\gamma}}\eps|u^x_\eps(t)|_{\rH}^2}\leq \ex{-\frac{\lambda_1}2 t}\ex{\frac{\bar{\gamma}}\eps|x|^2_{\rH}}+2,\;\;\; t>0.\]

\end{lemma}
\begin{remark}\label{rem-lem-exponential}

The result from Lemma \ref{lem-exponential} is also true for the stochastic Navier-Stokes equations with
multiplicative noise \eqref{eqn-SNSE-multiplicative} provided
$w$ is a cylindrical Wiener process on some separable Hilbert space $\rK$ and
\[
g:\rH \to R(\rK,\rH)
\]
is a continuous and bounded map. See also Remark \ref{rem-thm-exponential}.
\end{remark}

\begin{proof}[Proof of Lemma \ref{lem-exponential}]
Take arbitrary $\eps,\lambda, \gamma>0$, that will be specified later on. As a consequence of It\^{o}'s formula applied to the function
\[
\displaystyle{\vp: \mathbb{R}_+\times H \ni (t,u) \mapsto \ex{\lambda t} \,\ex{\frac \gamma\eps|u|_{\rH}^2} \in \mathbb{R},
}
\]
we have
\[\begin{array}{l}
\displaystyle{d\left[\ex{\frac \gamma\eps|u^x_\eps(t)|^2}\ex{\lambda t}\right]=\lambda \ex{\lambda t}\ex{\frac \gamma\eps|u^x_\eps(t)|^2}\,dt+\ex{\lambda t} \ex{\frac \gamma\eps|u^x_\eps(t)|^2}\frac{2\gamma}\eps \langle u^x_\eps(t),du^x_\eps(t)\rangle}\\
{\;}\\
\displaystyle{+ \ex{\lambda t}\ex{\frac \gamma\eps|u^x_\eps(t)|^2}\,\gamma\,\mbox{Tr}[Q Q^\star]\,dt
+\ex{\lambda t}\frac{2\gamma^2}{\eps}\ex{\frac\gamma\eps \,|u^x_\eps(t)|^2} |Q u^x_\eps(t)|_{\rH}^2\,dt}\\
{\;}\\
\displaystyle{=\ex{\lambda t}\ex{\frac \gamma\eps|u^x_\eps(t)|^2}\!\left[\left(-\frac {2\gamma}\eps\langle \rA u^x_\eps(t),u^x_\eps(t)\rangle+\gamma \,\mbox{Tr}[QQ^\star]-\frac{2\gamma}\eps \langle \rB( u^x_\eps(t),u^x_\eps(t)),u^x_\eps(t)\rangle\right.\right.}\\
{\;}\\
\displaystyle{\left.\left.+\lambda+\frac{2\gamma^2}{\eps}|Q u^x_\eps(t)|^2 \right)\,dt
+\frac{2\gamma}{\eps} \langle u^x_\eps(t),QdW(t)\rangle\right]}\\
{\;}\\
\displaystyle{\leq \ex{\lambda t}\ex{\frac \gamma\eps|u^x_\eps(t)|^2}\left[-\frac{2\gamma}\eps \left(\lambda_1-\frac{\gamma}2 \|Q\|^2\right) | u^x_\eps(t)|_{\rH}^2+\gamma \,\mbox{Tr}[QQ^\star]+\lambda\right]\,dt}\\
{\;}\\
\displaystyle{+\ex{\lambda t}\ex{\frac \gamma\eps|u^x_\eps(t)|^2}\frac{2\gamma}\eps\langle u^x_\eps(t),Qdw(t)\rangle.}
\end{array}\]
Now, if we put $\lambda=\lambda_1/2$ and choose (small enough) $\bar{\gamma}>0$  such that
\[\lambda_1-{\bar{\gamma}} \|Q\|^2\geq \frac {\lambda_1}2,\ \ \ \ \ \bar{\gamma} \,\mbox{Tr}[QQ^\star]\leq \frac {\lambda_1}2,\]
we get
\begin{eqnarray*}
\displaystyle{ d \left[ \ex{\frac{\lambda_1}2 t}\,\ex{\frac {\bar{\gamma}}\eps|u^x_\eps(t)|^2} \right]}
&\leq& \ds{
\frac{2\gamma}\eps \ex{\frac{\lambda}{2} t}\ex{\frac \gamma\eps|u^x_\eps(t)|^2}\langle u^x_\eps(t),Qdw(t)\rangle
}\\
{\; }&&\\
\displaystyle{}&+&\ds{\lambda_1\; \ex{\frac{\lambda_1}2 t}\,
\ex{\frac {\bar{\gamma}}\eps|u^x_\eps(t)|^2}\left[-\frac{\bar{\gamma}}\eps  | u^x_\eps(t)|_{\rH}^2+1\right]\,dt.}
\end{eqnarray*}

By taking expectation (and considering stopping times) this implies that

\[\begin{array}{l}
\displaystyle{ \mathbb{E} \left[ \ex{\frac{\lambda_1}2 t}\,  \,\ex{\frac {\bar{\gamma}}\eps|u^x_\eps(t)|^2} \right]}\\
\vs\\
\displaystyle{\leq  \ex{ \frac{\bar\gamma}{\eps} |x|^2 } + \lambda_1\, \mathbb{E} \left[\int_0^t\ex{\frac{\lambda_1}2 s}\,
\ex{\frac {\bar{\gamma}}\eps|u^x_\eps(s)|^2}\left[-\frac{\bar{\gamma}}\eps  | u^x_\eps(s)|_{\rH}^2+1\right]\,ds\right].}
\end{array}\]

Since $e^r(-r+1)\leq 1$, for any $r\geq0$, this yields
\[\ex{\frac{\lambda_1}2 t}\,\mathbb{E}\ex{\frac {\bar{\gamma}}\eps|u^x_\eps(t)|_{\rH}^2}\leq \ex{\frac{\bar{\gamma}}\eps|x|^2_{\rH}}+\lambda_1\int_0^t\ex{\frac{\lambda_1}2 s}\,ds=\ex{\frac{\bar{\gamma}}\eps|x|^2_{\rH}}+2\left(\ex{\frac{\lambda_1}2 t}-1\right),\]
so that
\[\mathbb{E}\ex{\frac {\bar{\gamma}}\eps|u^x_\eps(t)|_{\rH}^2}\leq \ex{-\frac{\lambda_1}2 t}\ex{\frac{\bar{\gamma}}\eps|x|^2_{\rH}}+2.\]

\end{proof}

Now, we continue with the proof of the main result in this section.

\begin{proof}[Proof of Theorem \ref{thm-exponential}] We use the notation introduced in the proof of Lemma \ref{lem-exponential}.
 Let us fix $R>0$ and $t> 0$.
By the previous  lemma and  Chebyshev's inequality,
 we have
\[\begin{array}{l}
\displaystyle{\mathbb{P}\big( u^x_\eps(t) \in B_{\rH}^{\rm c}(0,R)\big)= \mathbb{P}\left(|u^x_\eps(t)|_{\rH}>R\right)=\mathbb{P}\left(\ex{\frac {\bar{\gamma}}\eps|u^x_\eps(t)|_{\rH}^2}>\ex{\frac{R^2\bar{\gamma}}{\eps}}\right)}\\
\vs\\
\ds{\leq
\ex{-\frac{R^2\bar{\gamma}}{\eps}}\mathbb{E}\left(\ex{\frac {\bar{\gamma}}\eps|u^x_\eps(t)|_{\rH}^2}\right)\leq  \ex{-\frac{R^2\bar{\gamma}}{\eps}}\left[\ex{-\frac{\lambda_1}2 t}\ex{\frac{\bar{\gamma}}\eps|x|^2_{\rH}}+2\right].}
\end{array}\]
Now, due to the  ergocity of the invariant measure $\nu_\eps$, for any function $\varphi:\rH\to\mathbb{R}$, Borel and bounded,
\[\int_{\rH}\varphi(x)\,\nu_\eps(dx)=\lim_{T\to \infty}\frac 1T\int_0^T\mathbb{E}\,\varphi(u^0_\eps(s))\,ds.\]
This implies that for any $R>0$
\[\begin{array}{l}
\ds{ \nu_\eps \big( B_{\rH}^{\rm c}(0,R)\big) = \lim_{T\to \infty}\frac 1T\int_0^T\mathbb{P}\left(u^0_\eps(s) \in B_{\rH}^{\rm c}(0,R)\big) \right)\,ds}\delr{\nu_\eps\left(|x|_{\rH}>R\right)=\lim_{T\to \infty}\frac 1T\int_0^T\mathbb{P}\left(|u^0_\eps(s)|_{\rH}>R\right)\,ds}\\
{\;}\\
\ds{\leq \ex{-\frac{R^2\bar{\gamma}}{\eps}}\limsup_{T\to \infty}\frac 1T\int_0^T\left(\ex{-\frac{\lambda_1}2 s}+2\right)\,ds=2\,\ex{-\frac{R^2\bar{\gamma}}{\eps}}.}
\end{array}\]
Hence, if we fix $s>0$ and put
\[R_s:=\sqrt{\frac{2s}{\bar{\gamma}}},\ \ \ \ \eps_s:=\frac{R^2_s\bar{\gamma}}{2\log 2},\]
 we have that
\[\nu_\eps \big( B_{\rH}^{\rm c}(0,R_s)\big)\leq \ex{-\frac{R_s^2\bar{\gamma}}{2\eps}}=\ex{-\frac s\eps},\ \ \ \ \eps\leq \eps_s,\]
and this concludes proof of Theorem \ref{thm-exponential}.

\end{proof}

\begin{remark}
In the proof of \eqref{cb8} we have used the It\^o formula and for this reason we had to assume that the noise covariance $QQ^\star$ is of trace class. In fact, in order to prove \eqref{cb7}, we do not need to assume the covariance $QQ^\star$ to be of trace class, as shown below.

Suppose, as in \cite{Brz+Li_2006}, that $w$ is a cylindrical Wiener process on a Hilbert space $\rK$ such that the following  assumption,  already formulated  in \cite{Brz+Li_2006} (with some small difference) as Assumption 6.1,   is satisfied.

\begin{assumption}\label{ass:A-01}
$\mathrm{K} \subset \mathrm{H} $ is a Hilbert space
such that
 for some $\delta \in (0,1/2)$, the image of  $\rK$ by  $\mathrm{A}^{-\delta}$ is contained in $\mathrm{H} \cap \mathbb{L}^4$ and the map
\begin{equation}
\label{A-gamma}
\begin{aligned} \mathrm{A}^{-\delta}:K \to \mathrm{H} \cap \mathbb{L}^4
\hbox{ is } \gamma\hbox{-radonifying}.
\end{aligned}
\end{equation}

\end{assumption}

It is proven in    \cite[Remark 6.1]{Brz+Li_2006} that if $\rA$ is the Stokes operator on a $2$-dimensional torus,
then $\mathrm{A}^{-s}: \mathrm{H}\to \mathrm{H} \cap \mathbb{L}^4 $ is
$\gamma$-radonifying iff $s> \frac12$. In other words, if
$K=D(A^\alpha)$ for some $\alpha \in (0,\frac12]$, then   the map $\mathrm{A}^{-\delta}:K \to \mathrm{H} \cap
\mathbb{L}^4$ is $\gamma$-radonifying iff $\alpha+\delta>1/2$. Hence,  if
$K=D(A^\alpha)$ for some $\alpha \in (0,\frac12]$,
Assumption \ref{ass:A-01} is satisfied, with  $\delta \in (\frac12-\alpha,\frac12)$. Moreover,
if $K=D(A^\alpha)$ for some $\alpha > \frac12$, then
Assumption \ref{ass:A-01} is satisfied, for any  $\delta \geq 0$.

If we denote
\[z_\eps(t)=\sqrt{\eps}\int_0^t\ex{(t-s)\rA}d w^Q(s)=\sqrt{\eps}\,z_1(t),\ \ \ \ t\geq 0,\]
we have $u^x_\eps(t)=z_\eps(t)+v_\eps(t)$, where $v_\eps(t)$ is the solution of the problem
\[\frac{dv_\eps(t)}{dt}+\rA v_\eps(t)+\rB(v_\eps(t)+z_\eps(t),v_\eps(t)+z_\eps(t))=0,\ \ \ v_\eps(0)=x.\]
Then, we have
\[\frac{d}{dt}|v_\eps(t)|^2_{\rH}\leq \left(-\lambda_1+c\,\eps^2\,|z_1(t)|_{L^4(D)}^4\right)|v_\eps(t)|^2_{\rH}+c\,\eps^2\,|z_1(t)|_{L^4(D)}^4.\]
Now, we fix $M>1$ and define
\[g(r):=\log(r\vee M),\ \ \ \ \ r \in\,\mathbb{R}.\]
Then, since $g$ belongs to $H^1_{\rm loc}(\mathbb{R})$ and the weak derivative $g^\prime(r)$ equals $r^{-1}1_{\{r>M\}}$, we get
\[\begin{array}{l}
\ds{\frac d{dt}g\left(|v_\eps(t)|^2_{\rH}\right)=g^\prime\left(|v_\eps(t)|^2_{\rH}\right)\frac{d}{dt}|v_\eps(t)|^2_{\rH}}\\
\vs\\
\ds{\leq \frac 1{|v_\eps(t)|^2_{\rH}}c\,\eps^2\,|z_1(t)|_{L^4(D)}^4 1_{(M,\infty)} (|v_\eps(t)|^2_{\rH})}\\
{\;}\\
\ds{+ \frac 1{|v_\eps(t)|^2_{\rH}}\left( -\lambda_1+c\,\eps^2\,|z_1(t)|_{L^4(D)}^4\right)|v_\eps(t)|^2_{\rH}1_{(M,\infty)}(|v_\eps(t)|^2_{\rH})}\\
{\;}\\
\ds{\leq \frac {c\,\eps^2}M |z_1(t)|_{L^4(D)}^4+\left( -\lambda_1+c\,\eps^2\,|z_1(t)|_{L^4(D)}^4\right)1_{(M,\infty)}(|v_\eps(t)|^2_{\rH}).}
\end{array}\]
This implies
\[\begin{array}{l}
\ds{\mathbb{E}\,g\left(|v_\eps(t)|^2_{\rH}\right)-g(|x|^2_{\rH})}\\
\vs\\
\ds{\leq -\lambda_1\int_0^t\mathbb{P}\left(|v_\eps(s)|^2_{\rH}>M\right)\,ds+c_M\,\eps^2\int_0^t\mathbb{E}|z_1(s)|_{L^4(D)}^4\,ds.}
\end{array}\]
On the other hand,  if we choose $M>|x|_{\rH}^2$, then $g(|x|^2_{\rH})=g(0)$ and since   $g$ is an increasing function, we infer that
\[g\left(|v_\eps(t)|^2_{\rH}\right)-g(|x|^2_{\rH})\geq 0,\ \ \ \ \mathbb{P}-a.s.\]
Hence
\[\int_0^t\mathbb{P}\left(|v_\eps(s)|^2_{\rH}>M\right)\,ds\leq \frac{c_M}{\lambda_1}
\,\eps^2\int_0^t\mathbb{E}|z_1(s)|_{L^4(D)}^4\,ds.\]
Therefore, for any $R>\sqrt{2}\,\big( |x|_{\rH} \vee 1)$ we have, by the Chebyshev inequality, that
\[\begin{array}{l}
\ds{\int_0^t\mathbb{P}\left(|u^x_\eps(s)|^2_{\rH}>R^2\right)\,ds}\\
\vs\\
\ds{\leq \int_0^t\mathbb{P}\left(|v_\eps(s)|^2_{\rH}>\frac {R^2} 2\right)\,ds+\int_0^t\mathbb{P}\left(\eps\,|z_1(s)|^2_{\rH}>\frac {R^2}2\right)\,ds}\\
{\;}\\
\ds{\leq \big( \frac{c_M}{\lambda_1}+ \frac{2}{R^4}\big) \,\eps^2\int_0^t\mathbb{E}|z_1(s)|_{L^4(D)}^4\,ds,}
\end{array}\]
and this allows us to conclude that
\begin{eqnarray}
\nonumber\ds{\nu_\eps \big( B_{\rH}^{\rm c}(0,R)\big) \delr{\left(|x|_{\rH}>R\right)}}
&=&\ds{\lim_{T\to\infty}\frac 1T\int_0^T\mathbb{P}\left(|u^x_\eps(s)|^2_{\rH}>R^2\right)\,ds}
\\
&\leq& \ds{c_{R}\frac{\eps^2}T\int_0^T\mathbb{E}|z_1(s)|_{L^4(D)}^4\,ds.}
\label{cb48}
\end{eqnarray}

On the other hand, as in \cite{Brz+Li_2006}, by the Burkholder inequality, see \cite{Brz_1997-sc} and \cite{Ondr2004},  for any $s>0$ we have
\begin{equation}
 \label{equ:Eineq2}
\begin{array}{l}
\ds{ \mathbb{E}|z_1(s)|^4_{L^4} =
  \E \left| \int_{0}^s \ex{-(s-r)\rA} dw(r)\right|^4_{L^4} \leq c \big( \int_{0}^s
\| \ex{-r\rA}\|^2_{R(K, L^4)} dr \big)^2}\\
\vs\\
\ds{\leq c
\big(\int_0^\infty  \|\ex{-r\rA}\|^2_{R(K,L^4)} dr\big)^2= c
\big(\int_0^\infty  \|\rA^{\delta} \ex{-r\rA}\rA^{-\delta}\|^2_{R(K,L^4)} dr\big)^2}\\
\vs\\
\ds{\leq  c
\big(\int_0^\infty  \|\rA^{\delta} \ex{-r\rA}\|^2_{\mathcal{L}(L^4,L^4)} \| \rA^{-\delta}\|^2_{R(K,L^4)} dr\big)^2}\\
\vs\\
\ds{\leq  c \| \rA^{-\delta}\|^4_{R(K,L^4)}
\big(\int_0^\infty\!\!  \|\rA^{\delta} \ex{-r\rA}\|^2_{\mathcal{L}(L^4,L^4)}  dr\big)^2\leq c \| \rA^{-\delta}\|^4_{R(K,L^4)}
\big(\int_0^\infty  \frac{\ex{-2\mu r}}{r^{2\delta}}  dr\big)^2 .}
\end{array}
      \end{equation}
Thus, because by the ideal property of the $\gamma$-radonifying norm,  see e.g. Baxendale  \cite{Bax_1976} and \cite{Brz_1997-sc},
\[\|  T_1T_2 \|_{R(K,L^4)} \leq \| T_1 \|_{\mathcal{L}(L^4,L^4)} \|  T_2 \|_{R(K,L^4)},\]
we can conclude that
\[\sup_{s>0} \mathbb{E}\,|z_1(s)|^4_{L^4}<\infty.\]

\deln{
Since  we assume $\beta>1$, for any $s\geq 0$ we have
\[\mathbb{E}|z_1(s)|_{L^4(D)}^4\leq c\,\mathbb{E}|z_1(s)|_{D(\rA^{\frac 12})}^4\leq c\left(\sum_{k=1}^\infty \int_0^s \ex{-2\lambda_k(s-r)}\lambda_k^{1-\beta}\,dr\right)^2\leq c\,\sum_{k=1}^\infty\lambda_k^\beta<\infty.\]
}
Thanks to inequality \eqref{cb48}, this allows to conclude that   equality \eqref{cb7} holds.
Let us note that the idea of using the function $g$ goes back to  the proof of Theorem 5.4 in \cite{DaP+G_1995}.
\end{remark}

\bigskip

\section{Lower bounds}\label{sec-lower}

Our purpose here is proving the following upper bound.

\begin{theorem}
\label{thm-lower}
For any $\delta, \gamma >0$ and $\bar{x} \in\,\rH$, there exists $\eps_0>0$ such that
\begin{equation}
\label{cb60}
\nu_\eps \big( B_{\rH}^{\rm c}(\bar{x},\delta)\big)
\delr{\nu_\eps\left(  |x-\bar{x}|_\rH<\delta\right)}\geq \ex{-\frac{\rU(\bar{x})+\gamma}{\eps}},\ \ \ \ \eps\leq \eps_0.\end{equation}

\end{theorem}
Let us point out that in the proof of the result we will use that fact that $\nu_\eps$ is an
invariant measure of the Markov process corresponding to the stochastic Navier-Stokes

Before proceeding with the proof of Theorem \ref{thm-lower}, we need to prove a preliminary result.

\medskip

\begin{lemma}
\label{lemma6.2}
Suppose that  $\gamma, \bar{T}>0$, $\bar{x} \in\,\rH$ and $\bar{\varphi} \in\,L^2(0,\bar{T};\rH)$ satisfy
\begin{equation}
\label{cb50-1}\frac{1}2 |\bar{\varphi}|^2_{L^2(0,\bar{T};\rH)}\leq \rU(\bar{x})+\frac \gamma 4.
\end{equation}
Moreover, assume there exists  a solution $\bar{z} \in\,C([0,\bar{T}];\rH)$  to the problem
\begin{equation}
\label{cb50}
\bar{z}^{\,\prime}(t)+\rA \bar{z}(t)+B(\bar{z}(t),\bar{z}(t))=Q\bar{\varphi}(t),\ \ \ \bar{z}(0)=0,\ \ \bar{z}(\bar{T})=\bar{x}.
\end{equation}
Then, for all $\delta$ and $R>0$ there exists $T_0>0$  and $\varphi_0 \in L^2(0,T_0+\bar{T};\rH)$ such that
\begin{equation}
\label{cb100}
\frac 12\,|\varphi_0|^2_{L^2(0,T_0+\bar{T};\rH)}\leq \rU(\bar{x})+\frac\gamma 4
\end{equation}
and
\begin{equation}
\label{cb53}
\sup_{|x|_{\rH}\leq R}\left|z_x(\varphi_0)(T_0+\bar{T})-\bar{x}\right|_{\rH}\leq \frac \delta 2,
\end{equation}
where  $z_x(\varphi_0) \in C([0,T_0+\bar{T}];\rH)$ is the (unique) solution of the control problem
\begin{equation}
\label{cb52}
z^\prime(t)+\rA z(t)+B(z(t),z(t))=Q\varphi_0(t),\ \ t \in\,[0,T_0+\bar{T}],
\ \ \ \ z(0)=x.
\end{equation}
\end{lemma}

\begin{proof} Let us assume that $\gamma, \bar{T}>0$, $\bar{x} \in\,\rH$,  $\bar{\varphi} \in\,L^2(0,\bar{T};\rH)$
and $\bar{z} \in\,C([0,\bar{T}];\rH)$ satisfy the assumptions of our Lemma. Let us fix $\delta>0$ and $R>0$.

\medskip

Since, as it is well known,\footnote{For a proof, see e.g. \cite[Theorem 4.6]{Brz+Li_2006}.}   the solution of problem \eqref{cb50} depends continuously on the initial condition in $C([0,\bar{T}];\rH)$, we infer there exists $\rho>0$ such that
\begin{equation}
\label{cb75}
|y_0|_{\rH}\leq \rho\Longrightarrow |y_{y_0}(\bar{\varphi})(\bar{T})-\bar{x}|_{\rH}\leq \frac\delta 2,\end{equation}
where $y_{y_0}(\bar{\varphi}) \in\,C([0,\bar{T}];\rH)$ is the solution of the problem
\[y^\prime(t)+\rA y(t)+B(y(t),y(t))=Q\bar{\varphi}(t),\ \ t \in\,[0,\bar{T}],\ \ \ \ y(0)=y_0.\]

Now, let us consider a solution $u_x \in \,C([0,T_0];\rH)$ of  the homogenous Navier-Stokes equation
\begin{equation}
\label{cb55}
u^\prime(t)+\rA u(t)+\rB (u(t),u(t))=0,\ \ \ \ u(0)=x.
\end{equation}
According to \eqref{cb24}, we have
\[|u_x(t)|_{\rH}\leq \ex{-\lambda_1 t}|x|_{\rH},\ \ \ \ t\geq 0.\]
Hence, if we choose $T_0>0$ such that
$R \ex{-\lambda_1 T_0}\leq \rho,$
we have
\begin{equation}
\label{cb74}
\sup_{|x|_{\rH}\leq R}|u_x(T_0)|_{\rH}\leq \rho.\end{equation}
Next, let us   define a control $\varphi_0 \in L^2(0,T_0+\bar{T};\rH)$ by setting

\[\varphi_0(t)=\begin{cases}
0,& t \in\,[0,T_0],\\
\bar{\varphi}(t-T_0),  &  t \in\,[T_0,T_0+\bar{T}],
\end{cases}\]
and next let us  fix $ x \in \rH$ such that $|x|_{\rH}\leq R$. Then, the function $z \in\,C([0,T_0+\bar{T}];\rH)$  defined by
\[z(t)=\begin{cases}
u_x(t), &  t \in\,[0,T_0],\\
y_{u_x(T_0)}(\bar{\varphi})(t-T_0),  &  t \in\,[T_0,T_0+\bar{T}],
\end{cases}\]
is the unique solution to   problem
\[z^\prime(t)+\rA z(t)+\rB(z(t),z(t))=Q\varphi_0(t),\ \ \ \ t \in\,[0,T_0+\bar{T}],\ \ \ \ z(0)=x,\]
In particular we infer that  $z=z_x(\varphi_0)$. \\
Since
moreover $z_x(\varphi_0)(T_0+\bar{T})=y_{u_x(T_0)}(\bar{\varphi})(\bar{T})$ and $|x|_{\rH}\leq R$, due to \eqref{cb74} and \eqref{cb75},  we infer that
\[ |z_x(\varphi_0)(T_0+\bar{T})-\bar{x}|_{\rH}\leq \frac \delta 2.\]
This proves condition \eqref{cb53}. It remains to prove that $\varphi_0$ satisfies \eqref{cb100}. This however follows directly from the definition of $\varphi_0$ and assumption \eqref{cb50-1}.\\
\end{proof}

Now, we are ready to prove Theorem \ref{thm-lower}. As  we have proved Lemma \ref{lemma6.2},  its proof  is analogous to the proof of \cite[C.2]{sowers} and \cite[Theorem 6.1]{Cerrai+Roeckner_2005}.

\begin{proof}[Proof of Theorem \ref{thm-lower}] Let us  fix $\delta, \gamma >0$ and $\bar{x} \in\,\rH$.
Without loss of generality, we may  assume that $\rU(\bar{x})<\infty$.
Note, that in view of Theorem \ref{teo4.3}, this implies that $\bar{x}\in D(\rA^{\frac{1+\beta}2})$. Moreover,  by the definitions \eqref{eqn-B4-QP} for the quasi-potential $\rU$ and \eqref{eqn-B5} for the energy, we infer that
there exists  $\bar{T}>0$, a control $\bar{\varphi} \in\,L^2(0,\bar{T};\rH)$ and  a function $\bar{z} \in\,C([0,\bar{T};\rH)$ such that
 \[\frac 12|\bar{\varphi}|_{L^2(0,\bar{T};\rH)}^2\leq \rU(\bar{x})+\frac \gamma 4\]
and $\bar{z}$ is a solution to  the problem
\[z^\prime(t)+\rA z(t)+B(z(t),z(t))=Q\varphi(t),\ \ \ z(0)=0,\ \ \ z(T)=\bar{x}.\]

By  \eqref{cb7}
\delc{\[\lim_{\eps\to 0}\nu_\eps\left(|x|_{\rH}\leq R\right)=1.\]}
we can find $\bar{R}>0$, sufficiently large, and $\eps_1>0$ such that
\begin{equation}
\label{cb7+1}
\nu_\eps (B_{\rH}(0,\bar{R}) \geq 1-(1- \ex{-\frac{\gamma}{2}})=\ex{-\frac{\gamma}{2}}, \;\;\;\mbox{$\eps \in (0,\eps_1]$}.
\end{equation}
Note that trivially, the above implies that
\begin{equation}
\label{cb7+2}
\nu_\eps (B_{\rH}(0,\bar{R}) \geq \ex{-\frac{\gamma}{2\eps}}, \;\;\;\mbox{$\eps \in (0,1 \wedge \eps_1]$}.
\end{equation}

With all the data given and constructed, we can apply  Lemma \ref{lemma6.2} and  we can find
$T_0>0$  and $\varphi_0 \in L^2(0,T_0+\bar{T};\rH)$ such that
\begin{equation*}
\frac 12\,|\varphi_0|^2_{L^2(0,T_0+\bar{T};\rH)}\leq \rU(\bar{x})+\frac\gamma 4
\end{equation*}
and
\begin{equation*}
\sup_{|x|_{\rH}\leq R}\left|z_x(\varphi_0)(T_0+\bar{T})-\bar{x}\right|_{\rH}\leq \frac \delta 2,
\end{equation*}
where  $z_x(\varphi_0) \in C([0,T_0+\bar{T}];\rH)$ is the  solution of the control problem \eqref{cb52}.
Let us recall that for $x\in \rH$ and $\eps>0$,  the unique solution to the stochastic problem \eqref{eqn-SNSE-eps} is denoted by $u^x_\eps$.

Now, since by Theorem \ref{brc.teo.4.2},  the family $\{u^x_\eps\}_{\eps>0}$ satisfies the uniform large deviation principle in $C([0,T_0+\bar{T}];\rH)$,
  there exists $\eps_2>0$ such that for $|x|_{\rH}\leq \bar{R}$ and all $\eps \in (0,\eps_2]$,
\begin{equation}\label{ineq-ld}\mathbb{P}\left(|u^x_\eps-z_x(\varphi_0)|_{C([0,T_0+\bar{T}];\rH)}<\frac\delta2\right)\geq \ex{-\frac{|\varphi_0|^2_{L^2(0,T_0+\bar{T};\rH)}+\frac{\gamma}{2}}{2\eps}}\geq \ex{-\frac{\rU(\bar{x})+\frac{\gamma}{2}}\eps}.
\end{equation}

Let us  fix $x\in \rH$ such that $\vert x\vert_{\rH} \leq \bar{R}$. Then by inequality \eqref{cb53} we have

\[\begin{array}{l}
\ds{|u^x_\eps(T_0+\bar{T})-\bar{x}|_{\rH}\leq |u^x_\eps(T_0+\bar{T})-z_x(\varphi_0)(T_0+\bar{T})|_{\rH}+|z_x(\varphi_0)(T_0+\bar{T})-\bar{x}|_{\rH}}\\
{\;}\\
\ds{\leq |u^x_\eps(T_0+\bar{T})-z_x(\varphi_0)(T_0+\bar{T})|_{\rH}+\frac \delta 2.}
\end{array}\]
This implies that
\[
 |u^x_\eps(T_0+\bar{T})-z_x(\varphi_0)(T_0+\bar{T})|_{\rH}<\frac \delta 2 \then |u^x_\eps(T_0+\bar{T})-\bar{x}|_{\rH} < \delta.
\]

Therefore, since  $\nu_\eps$ is an invariant measure for the Markov process $u_x^\eps$,  we  infer that
\[\begin{array}{l}
\ds{\nu_\eps\big( B_{\rH}(\bar{x},\delta) \big)=\nu_\eps\left( |x-\bar{x}|_\rH<\delta\right)=\int_\rH\mathbb{P}\left(|u^x_\eps(T_0+\bar{T})-\bar{x}|_{\rH}<\delta\right)\,\nu_\eps(dx)}\\
{\;}\\
\ds{\geq \int_\rH\mathbb{P}\left(|u^x_\eps(T_0+\bar{T})-z_x(\varphi_0)(T_0+\bar{T})|_{\rH}<\frac\delta2\right)\,\nu_\eps(dx)}\\
{\;}\\
\ds{\geq \int_\rH\mathbb{P}\left(|u^x_\eps-z_x(\varphi_0)|_{C([0,T_0+\bar{T}];\rH)}<\frac\delta2\right)\,\nu_\eps(dx)}\\
{\;}\\
\ds{\geq \int_{B_{\rH}(0,\bar{R})}\mathbb{P}\left(|u^x_\eps-z_x(\varphi_0)|_{C([0,T_0+\bar{T}];\rH)}<\frac\delta2\right)\,\nu_\eps(dx).}
\end{array}\]
Applying \eqref{ineq-ld} we infer
that for $\eps \in (0,\eps_2]$,
\[\nu_\eps\big( B_{\rH}(\bar{x},\delta) \big)\geq \nu_\eps(B_{\rH}\big(0,\bar{R})\big)\ex{-\frac{\rU(\bar{x})+\frac{\gamma}{2}}\eps},\]
To conclude the proof, let us take $\eps_0:=\min\{1,\eps_1,\eps_2\}$. Then, by \eqref{cb7+2}, we infer that  for $\eps \in (0,\eps_0]$,
\[
\nu_\eps\big( B_{\rH}(\bar{x},\delta) \big) \geq  \ex{-\frac{\gamma}{2\eps}} \ex{-\frac{\rU(\bar{x})+\frac{\gamma}{2}}\eps}=\ex{-\frac{\rU(\bar{x})+{\gamma}}\eps}.
\]
This completes the proof of Theorem \ref{thm-lower}.
\end{proof}

\bigskip

\section{Upper bounds}\label{sec-upper1}

Let us recall here that $K_s$ is the level set of the quasipotential $\rU$, as defined in \eqref{eqn-K}, that is
\[K_s:=\{ x \in \rH: \rU(x) \leq s\}.\]

\begin{theorem}
\label{thm-upper}
For all $\delta, \gamma>0$ and   $s\geq 0$, there exists $\eps_0>0$ such that
\begin{equation}
\label{cb42}
\nu_\eps\left(\{ x \in \rH: \text{ dist}_\rH(x, K_s)\geq \delta\}\right)\leq \ex{-\frac{s-\gamma}{\eps}},\ \ \ \ \eps\leq \eps_0.\end{equation}
\end{theorem}

Before proceeding with the proof of Theorem \ref{thm-upper}, we state two auxiliary results, whose proofs are postponed till  next section.

\medskip

\begin{lemma}\label{Lem-CR-Lem7.1}
For all $\delta>0$ and $s>0$,   there exist  $\lambda=\lambda(\delta,s)>0$ and $\bar{T}=\bar{T}(\delta,s)>0$ such that for every $t\geq \bar{T}$ and  $z\in C([-t,0];\rH)$
\begin{equation}
\vert z(-t)\vert_{\rH}<\lambda,\ \ S_{-t}(z)\leq s\Longrightarrow
\text{dist}_{\rH}(z(0),K_s)< \delta.\end{equation}
\end{lemma}

\begin{lemma}\label{lem-73}
\label{lemma7.3}
For all $s,\delta, r>0$, there exists $\bar{n} \in\,\mathbb{N}$ such that
\[\beta_{\bar{n}}:=\inf\,\left\{S_{\bar{n}}(u)\ :\ u \in\,H_{r, s,\delta}(\bar{n})\right\}>s,\]
where
for each $n \in\,\mathbb{N}$, $s>0$, $\delta>0$ and $ r>0$, the set $H_{r, s,\delta}(n)$ is defined by
\begin{equation}\label{eqn-H_r s delta}
\ds{H_{r, s,\delta}(n):=\left\{ u \in\,C([0,n];\rH),\ |u(0)|_{\rH}\leq r,\ |u(j)|_{\rH}\geq \lambda,\ \ j=1,\ldots,n\right\},}
\end{equation}
and $\lambda$ is the constant depending on $s$ and $\delta$, obtained in Lemma \ref{Lem-CR-Lem7.1}.

\end{lemma}

 Assuming Lemmata \ref{Lem-CR-Lem7.1} and  \ref{lemma7.3}, the proof of Theorem \ref{thm-upper} follows the same
line of the proofs of \cite[C.3]{sowers} and  \cite[Theorem 7.3]{Cerrai+Roeckner_2005}.
We give here the proof, with some additional details, for the reader's convenience.

\begin{proof}[Proof of Theorem \ref{thm-upper}]

Let us  fix $\delta>0$, $\gamma>0$ and $s \geq 0$ and let us choose
positive constants $R_s$ and $\eps_s$,
as in Theorem \ref{thm-exponential}.
\medskip

Because  $\nu_\eps$ is an invariant measure  for the Markov process generated by equation \eqref{eqn-SNSE-eps} and
the set $\big\{ x \in\,\rH\ :\ \text{ dist}(x, K_s)\geq \delta\}$ is closed and hence a
Borel subset of $\rH$, we infer that
\begin{eqnarray}\label{eqn-zb-8.1}
   \nu_\eps\left( \big\{ x \in\,\rH\ :\ \text{ dist}(x, K_s)\geq \delta\}\right)&=&
   \int_{H}\mathbb{P}\left( \mbox{dist}\left(u^y_\eps(t),K_s\right)\geq \delta\right)\,\nu_\eps(dy) \\
&&\hspace{-4truecm}\lefteqn{=\int_{B_H^{\rm c}(0,R_s)}\mathbb{P}\left(\mbox{dist}\left(u^y_\eps(t),K_s\right)\geq \delta\right)\,\nu_\eps(dy)
}\nonumber\\
&&\hspace{-2truecm}\lefteqn{
+\int_{B_H(0,R_s)}\mathbb{P}\left(\mbox{dist}\left(u^y_\eps(t),K_s\right)\geq \delta\right)\,\nu_\eps(dy).}
\nonumber
\end{eqnarray}
Thanks to Theorem \ref{thm-exponential}, for any $\eps\leq \eps_s$ we have
\begin{eqnarray}\label{ineq-zb-8.2}
\int_{B_H^{\rm c}(0,R_s)}\mathbb{P}\left(\mbox{dist}\left(u^y_\eps(t),K_s\right)\geq \delta\right)\,\nu_\eps(dy)
&\leq & \ex{-\frac s\eps} .
\end{eqnarray}
Now, in view of Lemma \ref{lemma7.3}, we can pick  $\bar{n}\in \mathbb{N}$ such that
\[
u\in H_{R_s,s,\delta}(\bar{n})\Longrightarrow S_{\bar{n}}(u)  \geq s.
\]

Since the set $H_{R_s,s,\delta}(\bar{n})$ is  closed in  $C([0,n];\rH)$, and since by
Theorem \ref{brc.teo.4.2} the family $\{u^y_{\eps}\}_{\eps >0}$ satisfies the  large deviation principle in
$C([0,n];\rH)$ uniformly with respect to $y \in B_{\rH}(0,R_s)$, we infer that  there exists $\eps_1>0$ such that
\begin{eqnarray}\label{ineq-zb-8.3}
\sup_{y \in B_{\rH}(0,R_s)} \mathbb{P}\left(u^y_\eps \in\,H_{R_s,s,\delta}(\bar{n})\right)\leq \ex{-\frac{s-\gamma/2}\eps},\ \ \ \ \eps\leq \eps_1.
\end{eqnarray}
This implies that for $\eps \in (0,\eps_1)$,
\[
\int_{B_H(0,R_s)} \mathbb{P}\left(\mbox{dist}\left(u^y_\eps(t),K_s\right)\geq \delta,\
u^y_\eps \in H_{R_s,s,\delta}(\bar{n})\right) \,\nu_\eps(dy) \leq \ex{-\frac{s-\gamma/2}\eps}.
\]
Thus, for $\eps \in (0,\eps_1)$,
\begin{eqnarray}\label{ineq-zb-8.4}
\int_{B_H(0,R_s)} \mathbb{P}\left(\mbox{dist}\left(u^y_\eps(t),K_s\right)\geq \delta\right)
&\leq&    \ex{-\frac{s-\gamma/2}\eps}
\\
&&\hspace{-4truecm}{+\int_{B_H(0,R_s)} \mathbb{P}\left(\mbox{dist}\left(u^y_\eps(t),K_s\right)\geq
\delta,\ u^y_\eps \notin H_{R_s,s,\delta}(\bar{n})\right)\,\nu_\eps(dy)
.}
\nonumber
\end{eqnarray}
Thus we only have to deal with the second integral on the RHS of \eqref{ineq-zb-8.4}.

Let now fix $t \geq  \bar{n}$, $\eps \in (0, \eps_s \wedge \eps_1)$  and $y \in B_{\rH}(0,R_s)$.
In view of the definition of $H_{r,s,\delta}(\bar{n})$, we have that
\begin{equation}\label{eqn-H_r s delta2}
\left\{ u \in\,C([0,n];\rH),|u(0)|_{\rH} \leq  r\right\} \setminus H_{r, s,\delta}(n)=   \bigcup_{j=1}^n \left\{ u \in\,C([0,n];\rH),|u(j)|_{\rH}<  \lambda\right\}.
\end{equation}
Therefore,
because $\vert u^y_\eps(0)\vert_{\rH}= \vert y\vert_{\rH}\leq R_s$, we infer that
\begin{eqnarray*}
&&\hspace{-8truecm}\lefteqn{\Bigl\{\omega \in \Omega: \mbox{dist}\bigl(u^y_\eps(t),K_s\bigr)\geq \delta,\ u^y_\eps \notin H_{R_s,s,\delta}(\bar{n})\Bigr\}}
\\
&&\hspace{-5truecm}
\lefteqn{=\bigcup_{j=1}^n \Bigl\{\omega \in \Omega: \mbox{dist}\bigl(u^y_\eps(t),K_s\bigr)\geq \delta, |u^y_\eps(j)|_{\rH}< \lambda \Bigr\}.}
\end{eqnarray*}

Moreover, by the Markov property of the process $u^y_\eps$, we infer that if $P(\tau,t,dz)$, $0\leq \tau\leq t$, is the transition probability function corresponding to the Markov process
$u^y_\eps(t)$, $t\geq 0$ and $y \in \rH$, then
\begin{eqnarray}\label{ineq-zb-8.5}
\nonumber
&&\hspace{-4truecm}
\lefteqn{\mathbb{P}
\Bigl(\mbox{dist}\bigl(u^y_\eps(t),K_s\bigr)\geq \delta, |u^y_\eps(j)|_{\rH}< \lambda  \Bigr)}
\\
&&\hspace{-2truecm}
\lefteqn{= \int_{ \{|u^y_\eps(j)|_{\rH}< \lambda \} } P(j, y, dz) \mathbb{P}\bigl( \dist( u^z_\eps(t-j),K_s)\geq \delta \bigr)}
\\
&\leq& \sup_{|z|_{\rH}< \lambda} \mathbb{P}\bigl( \dist( u^z_\eps(t-j),K_s)\geq \delta \bigr).
\end{eqnarray}

Therefore,

\delc{Now, due to the Markov property, for any $t\geq \bar{n}$}
\begin{equation}
\label{cb41}
\begin{array}{l}
\hspace{-2truecm}\ds{\int_{\{|y|_{\rH}\leq R_s\}}\mathbb{P}\left(\mbox{dist}\left(u^y_\eps(t),K_s\right)\geq \delta,\ u^y_\eps \notin H_{R_s,s,\delta}(\bar{n})
\right)\,\nu_\eps(dy)}\\
\delc{{\;}\\
\hspace{3truecm}\ds{\leq \int_{|y|_{\rH}\leq R_s}\mathbb{P}\left(\bigcup_{j=1}^{\bar{n}}\left\{\mbox{dist}\left(u^y_\eps(t),K_s\right)\geq \delta,\
|u^y_\eps(j)|_{\rH}\leq \lambda\right\}\right)\,\nu_\eps(dy)}\\
{\;}\\}
\ds{\leq \sum_{j=1}^{\bar{n}}\sup_{z\in B_{\rH}(0,\lambda)}\mathbb{P}\left(\mbox{dist}\left(u^z_\eps(t-j),K_s\right)\geq \delta\right).}
\end{array}\end{equation}

Next, in order to estimate the RHS of the last equality, we fix $z\in B_{\rH}(0,\lambda)$ and  define two auxiliary
sets $ K_s(\lambda,t)$ and $\tilde{K}_s(z,t)$ by

\[K_s(\lambda,t):=\left\{u \in\,C([0,t];\rH)\ :\ S_t(u)\leq s,\ |u(0)|_{\rH}\leq \lambda\right\},\]
and
\[\tilde{K}_s(z,t):=\left\{u \in\,C([0,t];\rH)\ :\ S_t(u)\leq s,\ u(0)=z\right\}.\]	
Since, $\vert z\vert_{\rH} \leq \lambda$, we observe that
\[
\tilde{K}_s(z,t) \subset K_s(\lambda,t).
\]

Moreover, according to  Lemma \ref{Lem-CR-Lem7.1},  there exists $\bar{T}>0$ such that for any $T\geq \bar{T}$
\[\varphi \in\,K_s(\lambda,T)\Longrightarrow \mbox{dist}(\varphi(T),K_s)\leq \frac \delta 2.\]
In what follows we fix $t\geq  \max\{\bar{T},\bar{n}\}$, and we prove that for any $u \in\,C([0,t];\rH)$  such that
\begin{equation}
\label{uff1}
\dist_{C([0,t];\rH)}(u, K_s(\lambda,t))<\frac\delta2,
\end{equation}
we have
 \[\dist(u(t),K_s) <\delta.\]
Indeed,  if
\eqref{uff1} holds, then there exists $\varphi \in\,K_s(\lambda,t)$
such that
\[\dist_{C([0,t];\rH)}(u, \varphi) <\frac\delta2,\]
so that  $\vert u(t)-\varphi(t)\vert_{\rH} <\frac\delta2$.
Hence, by the triangle inequality, we infer that
\[
\dist(u(t),K_s) \leq \vert u(t)-\varphi(t)\vert_{\rH}+\dist(u(t),K_s)< \frac\delta2+\frac\delta2=\delta.
\]

Since, $\tilde{K}_s(z,t) \subset K_s(\lambda,t)$  we deduce that
\[\begin{array}{l}
\ds{\mathbb{P}\left(\mbox{dist}\left(u^z_\eps(t),K_s\right)\geq \delta\right)
\leq \mathbb{P}\left(\mbox{dist}_{C([0,t];\rH)}\left(u^y_\eps,K_s(\lambda,t)\right)>\frac \delta 2\right)}
\\\ds{\leq
\mathbb{P}\left(\mbox{dist}_{C([0,t];\rH)}\left(u^z_\eps,\tilde{K}_s(z,t)\right)>\frac \delta 2\right).}
\end{array}\]

\delc{
Therefore, by the triangular inequality
\[\mbox{dist}\left(u^y_\eps(t),K_s\right)\geq \delta\Longrightarrow
\mbox{dist}\left(u^y_\eps(t),\varphi(t)\right)\geq \frac \delta 2\Longrightarrow
\mbox{dist}_{C([0,t];\rH)}\left(u^y_\eps,K_s(\lambda,t)\right)\geq \frac \delta 2.\]
As we are assuming $|y|_{\rH}\leq \lambda$, this implies

\[\begin{array}{l}
\ds{\mathbb{P}\left(\mbox{dist}\left(u^y_\eps(t),K_s\right)\geq \delta\right)
\leq \mathbb{P}\left(\mbox{dist}_{C([0,t];\rH)}\left(u^y_\eps,K_s(\lambda,t)\right)>\frac \delta 2\right)}\\
{\;}\\
\ds{\leq
\mathbb{P}\left(\mbox{dist}_{C([0,t];\rH)}\left(u^y_\eps,K_s(y,t)\right)>\frac \delta 2\right).}
\end{array}\]
}

Next, as  the set \\
\[\Bigl\{u \in  C([0,t];\rH) : \dist_{C([0,t];\rH)}\left(u_\eps,K_s(y,t)\right)\geq \frac \delta 2\Bigr\}\]
is  closed in  $C([0,t];\rH)$, and,  by Theorem \ref{brc.teo.4.2}, the family $\{u^z_{\eps}\}_{\eps >0}$ satisfies
the  large deviation principle in $C([0,t];\rH)$ uniformly with respect to $z \in B_{\rH}(0,\lambda)$, we infer that  there exists $\eps_2(t)>0$
such that
\delc{As the family $\{u^y_\eps\}_{\eps>0}$ satisfies the  large deviation principle in $C([0,t];\rH)$,
uniformly with respect to $|y|_{\rH}\leq \lambda$,  see Theorem \ref{brc.teo.4.2},  there exists $\eps(t)>0$ such that}
\[\sup_{z \in B_{\rH}(0,\lambda)}\mathbb{P}\left(\mbox{dist}_{C([0,t];\rH)}\left(u^z_\eps,K_s(y,t)\right)>\frac \delta 2\right)\leq \ex{-\frac{s-\gamma/2}{\eps}},\ \ \ \ \eps\leq \eps(t).\]
Therefore, if we define
\[\eps_3:=\min\{\eps_2(t-1),\ldots,\eps_2(t-\bar{n}), \eps_s,\eps_1\},\]
due to \delc{\eqref{cb40}} \eqref{eqn-zb-8.1}, \eqref{ineq-zb-8.2}, \eqref{ineq-zb-8.4}  and \eqref{cb41}, we deduce that  for $\eps\leq \eps_3$,
\[\begin{array}{l}
\ds{\nu_\eps\left( x \in\,\rH\ :\ \text{ dist}(x, K(s))\geq \delta\right)}\\
\vs\\
\ds{\leq
\ex{-\frac s\eps}+(1+\bar{n})\ex{-\frac{s-\gamma/2}\eps}=\ex{-\frac{s}\eps}\bigl( 1+(1+\bar{n})\ex{-\frac{\gamma/2}\eps} \bigr).}
\end{array}\]
This clearly implies \eqref{cb42}, if we take $\eps_0$ sufficiently small.

\end{proof}

\section{Proof of Lemmata \ref{Lem-CR-Lem7.1} and \ref{lem-73}}\label{sec-upper2}

\begin{proof}[Proof of Lemma \ref{Lem-CR-Lem7.1}.] Suppose  that there exist
 $\delta>0$ and $ s>0$   such that for every $n \in \mathbb{N}$ there exists a function
$z_n\in C([-n,0];H)$
with
\begin{equation}
\label{cb461}
S_{-n}(z_n)\leq s,\ \ \ \ \text{dist}_{\rH}(z_n(0),K_s)\geq  \delta,\end{equation}
and
\begin{equation}
\label{cf460}
\beta_n:=\vert z_n(-n)\vert_H^2 \todown 0,\ \ \ \text{as}\ n\to\infty.\end{equation}
We will show that this leads to a contradiction.

Note that for every $n\in \mathbb{N}$, the function $z_n$ satisfies the following a priori inequality
\begin{equation}
\label{bc2}
\sup_{s \in [-n,0]} \vert z_n(s)\vert_H^2 +\int_{-n}^0 \vert z_n(s) \vert_V^2\,ds  \leq \vert z_n(-n)\vert_H^2 + \int_{-n}^0 \vert f_n(s)\vert^2_{V^\prime}\,ds,
\end{equation}
where
\[f_n(s):=  {\mathcal H}(z_n)(s)=z_n^\prime(s)+\rA z_n(s) +\rB(z_n(s),z_n(s)),\;\; s\in (-n,0). \]

Therefore, in view of inequality \eqref{ineq-useful}, by  conditions \eqref{cb461} and \eqref{cf460}, we infer that there exists $c>0$ such that
\[
\sup_{s \in [-n,0]} \vert z_n(s)\vert_H^2 +\int_{-n}^0 \vert z_n(s) \vert_V^2\,ds  \leq \beta_n+c\,s\delr{\frac{2s}{\lambda_1^{1+\beta}}}.
\]
Moreover, by inequality \eqref{brc82}, there exists a constant $c>0$ such that
\begin{equation}
\label{bc1}
\begin{array}{l}
\ds{|\sqrt{\cdot+n}\,z_n(\cdot)|^2_{L^\infty(-n,0;\rV)}+|\sqrt{\cdot+n}\,z_n(\cdot)|^2_{L^2(-n,0;D(\rA))}}\\
{\;}\\
\ds{\leq c\exp\,\left[c\big( \vert z_n(-n)\vert_{\rH}^4+ \vert f_n\vert^4_{L^2(-n,0;\rV^\prime)} \big)\right]}\\
\vs\\
\ds{\times \Big( \vert z_n(-n)\vert_{\rH}^2+ \vert f_n\vert^2_{L^2(-n,0;\rV^\prime)}+n\vert f_n\vert^2_{L^2(-n,0;\rH)} \Big).}
\end{array}
\end{equation}

If $s\in [-\frac{n}2,0]$, we have $s+n \geq \frac{n}2$ and therefore, from \eqref{bc1}, we get
\begin{equation*}
\begin{array}{l}
\ds{\frac{n}2\Big( |z_n|^2_{L^\infty(-\frac{n}2,0;\rV)}+|\,z_n|^2_{L^2(-\frac{n}2,0;D(\rA))}\Big)}\\
{\;}\\
\ds{\leq c\exp\,\left[c\big( \vert z_n(-n)\vert_{\rH}^4+ \vert f_n\vert^4_{L^2(-n,0;\rV^\prime)} \big)\right]}\\
\vs\\
\ds{\times \Big( \vert z_n(-n)\vert_{\rH}^2+ \vert f_n\vert^2_{L^2(-n,0;\rV^\prime)}+n\vert f_n\vert^2_{L^2(-n,0;\rH)} \Big).}
\end{array}
\end{equation*}
This implies that there exists a constant $c_2=c_2(s,\vert z_1(-1)\vert_H^2)$ such that
\begin{equation}
\label{brc82bis}
 |z_n|_{L^\infty(-\frac{n}2,0;\rV)}+|\,z_n|_{L^2(-\frac{n}2,0;D(\rA))} \leq
c_2, \;\; n\in \mathbb{N}.
\end{equation}

Moreover, this implies that there exists a  constant $c_3>0$ such that
\begin{equation}
\label{cb83}
\vert z_n^\prime(\cdot)\vert_{L^2(-\frac n2,0;\rH)}\leq c_3,
\end{equation}
Indeed, since for $\;\; n\in \mathbb{N}$,
\[z_n^\prime(t)=f_n(t)-\rA z_n(t)-B(z_n(t),z_n(t)),\ \ \ \ t \in\,(-\frac n2,0),\]
we infer  that
\begin{equation}
\label{cb6}
\begin{array}{l}
\ds{|z_n^\prime|_{L^2(-\frac n2,0;\rH)}}\\
{\;}\\
\ds{\leq |f_n|_{L^2(-\frac n2,0;\rH)}+|\rA z_n|_{L^2(-\frac n2,0;\rH)}+|B(z_n,z_n)|_{L^2(-\frac n2,0;\rH)}}\\
{\;}\\
\ds{\leq c\,f_n\vert^2_{L^2(-\frac n2,0;\rH)} +c\,|z_n|^2_{L^2(-\frac n2,0;D(\rA))}+|B(z_n,z_n)|_{L^2(-\frac n2,0;\rH)}+1.}
\end{array}\end{equation}
Next, from inequality  \eqref{ineq-B01},  we  deduce that
\[\ds{\vert\rB(z_n,z_n)\vert_{L^2(-\frac n2,0;\rH)}^2\leq \vert z_n\vert_{L^\infty(-\frac n2,0;\rH)}\vert z_n\vert_{L^\infty(-\frac n2,0;\rV)}\vert z_n\vert _{L^2(-\frac n2,0;\rV)}\vert z_n\vert_{L^2(a,b;D(\rA))}.}\]
Thanks to \eqref{bc2} and \eqref{brc82bis}, this implies that for some constant $c>0$
\[\begin{array}{l}
\ds{\vert \rB(z_n,z_n)\vert_{L^2(-\frac n2,0;\rH)}\leq c,\;\;\; \;\; n\in \mathbb{N}.}
\end{array}\]
Hence inequality \eqref{cb83} follows, due to \eqref{cb6} and \eqref{brc82bis}.

Now, let us fix $k\in\mathbb{N}$. Notice that if $n\geq 2k$ then $[-k,0]\subset [-\frac{n}2,0]$.  We can consider the sequence $\{z_n\}_{n=2k}^\infty$, or more precisely, the sequence of restrictions of that sequence to the time interval $[-k,0]$. According to  \eqref{bc2}, \eqref{brc82bis} and \eqref{cb83}, this sequence satisfies
\[\sup_{r \in [-k,0]}\vert z_n(r)\vert_H^2 +\int_{-k}^0 \vert z_n(r) \vert_V^2\,dr  \leq \beta_n+c\,s\delr{\frac{2c}{\lambda_1^{1+\beta}}}\]
and
\[ \vert z_n\vert^2_{L^\infty(-k,0;\rV)}+|\,z_n|^2_{L^2(-k,0;D(\rA))} +\vert z_n^\prime\vert_{L^2(-k,0;\rH)}
 \leq c.\]

 Moreover, as $S_{-n}(z_n)\leq s$, for any $n \in\,\mathbb{N}$ we have
 \[|f_n|_{L^2(-k,0;D(\rA^{\frac{\alpha}{2}}))}\leq \sqrt{2s}.\]
Hence, by a standard compactness argument, (compare with the first method of proof of Theorem 6.1 in \cite[page 71 onwards]{Lions_1969} and the proof of Theorem III.3.10 from \cite{Temam_2001}),  for each fixed $k \in\,\mathbb{N}$ there exist two subsequences
\[ \{z^{k}_{n_j}\}_{j=1}^\infty\subseteq \{z_n\}_{n=2k}^\infty \mbox{ and } \{f^{k}_{n_j}\}_{j=1}^\infty\subseteq \{f_n\}_{n=2k}^\infty,\]
 and  two functions  $f^k \in L^2(-k,0,\rH)$ and $u^k \in  C([-k,0];\rV) \cap L^2(-k,0;D(\rA))$, with
$ D_t u^k \in L^2(-k,0;\rH) $, such that, as $j\to\infty$,
\[
z^k_{n_j} \to u^k, \mbox{ weakly in } L^2(-k,0;D(\rA))\ \mbox{and strongly in } L^2(-k,0;\rV)\cap C([-k,0];\rH),\]
and
\[f^k_{n_j} \to f^k, \mbox{ weakly in } L^2(-k,0;D(\rA^{\frac{\alpha}{2}})),\]
 $u^k$ satisfies
\[D_t{u^k}+\rA u^k +\rB(u^k,u^k)=f_k \mbox{ on } (-k,0),
\]
and
\[
\vert u^k\vert_{C([-k,0];\rH)}^2 \leq c\,s\delr{\frac{2s}{\lambda_1^{1+\beta}}},\ \ \ \ \ \ \vert f^k\vert^2_{L^2(-k,0;D(\rA^\frac{\alpha}{2}))} \leq 2 s.
\]
Moreover, by an inductive argument, the sequences  $\{z^{k}_{n_j}\}_{j=1}^\infty$ and $\{f^{k}_{n_j}\}_{j=1}^\infty$ can be chosen in such a way  that, for any $k \in\,\mathbb{N}$, the restrictions of $u^{k+1}$ and $f^{k+1}$ to $(-k,0)$ are equal to $u^k$ and  $f^k$, respectively.

This allows us  to define two functions $u$ and $f$ on the interval $(-\infty,0)$ such that for every $k$, the restrictions of $u$ and $f$ to $(-k,0)$ are equal to $u^k$ and $f^k$, respectively. These functions $u$ and $f$ satisfy, for every $k \in\,\mathbb{N}$,
\[D_t{u}+\rA u +\rB(u,u)=f \mbox{ on } (-k,0),
\]
and
\[
\vert u\vert_{C([-k,0];\rH)}^2 \leq c\,s\delr{\frac{2s}{\lambda_1}},\ \ \ \ \  \vert f\vert^2_{L^2(-k,0;D(\rA^{\frac{\alpha}{2}}))} \leq 2 s.
\]
The last of these properties implies that  $f\in L^2(-\infty,0;D(\rA^{\frac\beta 2}))$ and
\[
\vert f\vert^2_{L^2(-\infty,0;D(\rA^{\frac\beta 2}))} \leq 2s.
\]
Moreover
\[D_t{u}+\rA u +\rB(u,u)=f \mbox{ on } (-\infty,0),
\]
so that
 $S_{-\infty}(u)\leq s$. Finally, due to \eqref{cf460},  there exists a sequence $\{t_n\}\downarrow -\infty$ such that
 \[\lim_{n\to\infty}|u(t_n)|_{\rH}=0.\]
 Therefore, by Lemma \ref{lemma4.2} we infer that $u \in\,\mathcal{X}$. Thus, thanks to the characterization of $\rU$ given in part (2) of Theorem \ref{teo4.3}, we can conclude that $\rU(u(0))\leq s$,  so that $u(0)\in K_s$.

On the other hand,
\[\lim_{j\to\infty}z^1_{n_j}=u,\ \ \ \ \ \mbox{in}\ C([-1,0];\rH),\] and, by our assumptions, $\text{dist}_{\rH}(z_n(0),K_s)\geq \delta$. Hence
 \[\text{dist}_{\rH}(u(0),K_s)\geq \delta\]
which contradicts the  fact that $u(0) \in\,K_s$.
 \end{proof}

\begin{proof}[Proof of Lemma \ref{lem-73}.]
Let us assume that there exist $s>0$, $\delta>0$ and $r>0$ such that for every $n \in\,\mathbb{N}$ there exists  $u_n \in\,H_{r, s,\delta}(n)$ such that
\[S_n(u_n)\leq s+1.\]
In particular, due to \eqref{cb3}, we have
\begin{equation}
\label{cb22}
|u_n|_{C([0,n];\rH)}\leq c_{\sqrt{2(s+1)}}\left(1+r\right)=:c_{s,r},\ \ \ n \in\,\mathbb{N}.\end{equation}
Now, for any $k \in\,\mathbb{N}$, we define
\[\gamma_k:=\inf\left\{S_k(u)\,;\,u \in\,C([0,k];\rH),\ |u(0)|_{\rH}\leq c_{s,r}\wedge r,\ |u(k)|_{\rH}\geq \lambda\right\}.\]
If we show that there exists $\bar{k} \in\,\mathbb{N}$ such that $\gamma_{\bar{k}}>0$, then, due to  \eqref{cb22}, we have
\[S_{n\bar{k}}(u_{n\bar{k}}) \geq n\,\gamma_{\bar{k}},\ \ \ \ n \in\,\mathbb{N},\]
which contradicts the fact that $S_{n\bar{k}}(u_{n\bar{k}})\leq s+1$.
Therefore, in order to conclude our proof, we show that  there exists some $\bar{k} \in\,\mathbb{N}$ such that $\gamma_{\bar{k}}>0$.

For any $x \in\,\rH$, we denote by $z_x(t)$ the solution of the problem
\[z_x^\prime(t)+\rA z_x(t)+B(z_x(t),z_x(t))=0,\ \ \ \ z_x(1)=x.\]
According to \eqref{cb24},  there exists some integer $\bar{k}\geq 1$ such that
\begin{equation}
\label{cb27}
|x|_{\rH}\leq c_{s,r}\Longrightarrow |z_x(t)|_{\rH}\leq \frac \lambda 2,\ \ \ \ t\geq \bar{k}.\end{equation}
We show that, for such $\bar{k}$, it holds $\gamma_{\bar{k}}>0$. Actually, if $\gamma_{\bar{k}}=0$, then there exists a sequence
\[\{v_n\}_{n \in\,\mathbb{N}}\subset \left\{ u \in\,C([0,\bar{k}];\rH)\,;\,|u(0)|_{\rH}\leq c_{s,r}\wedge r,\ |u(k)|_{\rH}\geq \lambda\right\},\]
such that
\begin{equation}
\label{cb33}
\lim_{n\to\infty} S_{\bar{k}}(v_n)=0.
\end{equation}
Thus, there exists $\bar{n} \in\,\mathbb{N}$ such that
$ S_{\bar{k}}(v_n)\leq s+1$, for any $n\geq \bar{n}$ and hence, according to \eqref{cb22}, $|v_n(1)|_{\rH}\leq c_{s,r}$, for any $n\geq \bar{n}$.
Moreover, thanks to \eqref{brc82}, there exists a constant $\tilde{c}_{s,r,\bar{k}}$ such that \begin{equation}
\label{cb20}
|v_n|_{L^\infty(1,\bar{k};\rV)}\leq \tilde{c}_{s,r,\bar{k}},\ \ \ \ n\geq \bar{n}.
\end{equation}
This means, in particular,  that there exists a subsequence $\{v_{n_j}\}_{j \in\,\mathbb{N}}\subset \{v_n\}_{n \in\,\mathbb{N}}$ and $\bar{x} \in\,\rH$ such that
\begin{equation}
\label{cb34}
\lim_{j\to\infty }|v_{n_j}(1)-\bar{x}|_{\rH}=0.\end{equation}
Since $|v_n(1)|_{\rH}\leq c_{s,r}$, it follows that $|\bar{x}|_{\rH}\leq c_{s,r}$, and then, due to \eqref{cb27}, it follows that
\begin{equation}
\label{cb30}
|z_{\bar{x}}(\bar{k})|_{\rH}\leq \frac \lambda 2.
\end{equation}
Now, as a consequence of \eqref{cb33}, for every $n \in\,\mathbb{N}$ there exists  $f_n \in\,L^2(0,\bar{k};\rH)$ such that
\[v_n^\prime(t)+\rA v_n(t)+\rB(v_n(t),v_n(t))=f_n(t),\]
and \[\lim_{n\to\infty}|f_n|_{L^2(0,\bar{k};\rH)}=0.\]
According to \eqref{cb34}, this implies that
\[\lim_{j\to\infty}|v_{n_j}-z_{\bar{x}}|_{C([1,\bar{k}];\rH)}=0,\]
so that $|z_{\bar{x}}(\bar{k})|_{\rH}\geq \lambda$, which contradicts \eqref{cb30}.

\end{proof}

\bigskip


\appendix

\section{Exponential estimates}\label{app-exponential estimates}

Let  us first present a small generalization of the exponential estimates from \cite{Brz+Peszat_2000}.
We adopt here the same notations and assumptions from that paper, i.e.  we assume  that $X$ is a real, separable Banach space satisfying the
following condition.
\begin{trivlist}
\item[$(\mathbf{H})$] There exist  $p\ge 2$ such that
the  function \[\varphi :X\ni x \mapsto\vert  x\vert_X ^p\in \mathbb{R}\] is of
$C^2$ class and there are $k_1,k_2>0$ such that
\[\vert \varphi^\prime(x)\vert_{X^\star} \le k_1\vert  x\vert _X^{p-1} \mbox{ and }
\vert \varphi^{\prime\prime}(x)\vert_{L(X\times X)}\le k_2\vert  x\vert_X ^{p-2} \mbox{  for every } x\in X.\]
\end{trivlist}
Note that $L^q$-spaces, with $q\ge 2$,  satisfy the condition $(\mathbf{H})$, and the "good" values of $p$ are all $p \geq q$.

The next assumption is strengthened, compared to \cite{Brz+Peszat_2000}.

\begin{trivlist}
\item[$(\mathbf{A})$]
$\{S(t)\}_{t\ge 0}$
is a $C_0$-semigroup  on $X$ such that for a certain $\mu >0$,
\[
\Vert S(t)\Vert_{L(X)} \leq e^{-\mu t} , \;\;\; t\geq 0.
\]
\end{trivlist}
The generator of the semigroup will be denoted by $-A$.

Note that in \cite{Brz+Peszat_2000} the generator of the semigroup in question was denoted by $A$.
Note also that if the semigroup $\{S(t)\}_{t\ge 0}$ satisfies the above assumption $(\mathbf{A})$, then the semigroup $\{e^{\mu t}S(t)\}_{t\ge 0}$ (whose generator is $\mu I-A$) is a contraction semigroup and hence satisfies the assumptions from \cite{Brz+Peszat_2000}.

\begin{trivlist}
\item[$(\mathbf{P})$]
A system $\mathfrak{U}=(\Omega, \mathcal{ F},\mathbb{F},
\mathbb{P})$, where $\mathbb{F}=(\mathcal{ F}_t)_{t\ge 0}$, is  a filtered
probability space. \\
A family $W(t)$, $t\ge 0$, of bounded linear
operators from a real separable Hilbert
space  $\rK$ into $L^2(\Omega ,\mathcal{ F},\mathbb{P})$, defines
a cylindrical Wiener process on $\rK$, i.e.
\begin{trivlist}
\item[(i)] for all $t\ge 0$ and $h_1,h_2\in \rK$, it holds
$\mathbb{E}\,[W(t)h_1 W(t)h_2] = t\, \langle h_1 ,h_2 \rangle$,
\item[(ii)]
for any $h\in \rK$, $(W(t)h)_{t\ge 0}$ is a real-valued
$(\mathcal{ F}_t)$-adapted  Wiener process.
\end{trivlist}
\end{trivlist}
Before we  give the last  assumptions, let us recall that  we denote  by $M(\rK,X)$ the Banach space of all
$\gamma$-radonifying linear mappings $L$ acting from $\rK$ into $X$, see
\cite{Brz_1995} or \cite{Brz_1997-sc} for more details. The norm in the space $M(\rK,X)$
will be denoted by $\Vert \cdot \Vert$.

\begin{trivlist}
\item[$(\mathbf{S})$]
The process  $\xi=\big(\xi(t), t \ge 0\big)$ is an $M(\rK,X)$-valued $\mathbb{F}$-progressively measurable
process such that
\[
\Vert \xi\Vert_{\infty}:= \esssup_{t \geq 0}
\Vert \xi(t)\Vert<\infty.\]
\end{trivlist}

Now, let us define
\begin{equation}
\label{E11}
u(t):=\int^t_0 S(t-r)\xi (r) \d W(r),\qquad t\ge 0.
\end{equation}
This means that $u$ is a mild solution to

\begin{equation}\label{E121}
\d u(t)+Au(t)\d t=\xi (t)\d W(t), \qquad u(0)=0.
\end{equation}

If we denote
\[
y(t)=\ex{\mu t} u(t), \;\; t \geq 0,
\]
by an easy application of an  It\^o formula, we have
\begin{eqnarray*}
dy(t)&=&\mu \ex{\mu t} u(t) \, dt+ \ex{\mu t} \, d u(t)
\\
&=&\mu  y(t) \, dt+ \ex{\mu t} \, \big[ -Au(t)\d t+\xi (t)\d W(t)  \big] \\
&=& \mu  y(t) \, dt -A [ \ex{\mu t} \,  u(t) ] \d t+\ex{\mu t} \xi (t)\d W(t),
\end{eqnarray*}
so that
\[
dy(t)+ (A(t)-\mu I) y(t) \d t=  \ex{\mu t} \xi(t)\d W(t),
\]
Next, if we introduce an auxiliary process $\eta$ defined by
\[
\eta (t) = \ex{\mu t} \xi(t), \;\; t\ge 0,
\]
we have
\begin{equation}
\label{E11'}
y(t):=\int^t_0 e^{\mu (t-s)}S(t-r)\eta (r) \d W(r),\qquad t\ge 0.
\end{equation}

Since the semigroup $\{e^{\mu t}S(t)\}_{t\ge 0}$ (whose generator is $\mu I-A$) is a contraction semigroup,   by Theorem 1.2 
from  \cite{Brz+Peszat_2000} we infer that
there exists a universal constant $c>0$ such that for all $T>0$ and $R>0$
\begin{equation}\label{E114}
\mathbb{P}\Bigl(\, \sup_{0\le t\le T}\vert  y(t)\vert_X \ge R\Bigr)
\le 3 \exp\Bigl\{-\frac {R^2}{4cM_T}\Bigr\}
\end{equation}
where
\[
M_T:= \esssup_{\Omega} \int^T_0\Vert \eta (t)\Vert ^2\d t.
\]

By assumption (S) we have
\[
\int^T_0\Vert \eta (t)\Vert ^2\d t=\int^T_0 \ex{2 \mu t}\Vert  \xi(t) \Vert ^2 \d t   \leq  \frac{\Vert  \xi \Vert_{\infty} ^2}{2\mu T} \big (\ex{2 \mu T} -1\big)
\leq  \frac{\Vert  \xi \Vert_{\infty} ^2}{2\mu T}  \ex{2 \mu T}.
\]
Moreover, for every $r>0$, we have
\[
\mathbb{P}\Bigl(\, \vert  u(T)\vert_X \ge r\Bigr)=
\mathbb{P}\Bigl(\, \vert  y(T)\vert_X \ge \ex{\mu T} r\Bigr)
\leq
\mathbb{P}\Bigl(\, \sup_{0\le t\le T}\vert  y(t)\vert_X \ge  \ex{\mu T} r\Bigr)
\]
Thus, by applying inequality \eqref{E114} with
\[
M_T=\frac{\|\xi\|_\infty^2}{2\mu T}  \ex{2 \mu T}  \mbox{ and } R= \ex{\mu T} r
\]
we infer that
\begin{eqnarray*}
\mathbb{P}\Bigl(\, \vert  u(T)\vert_X \ge r\Bigr)
&\leq& 3 \exp\Bigl\{-\frac {R^2}{4c M_T}\Bigr\}=3  \exp\Bigl\{-\frac {2 \mu T  \ex{2 \mu T}  r^2}{4c \Vert  \xi \Vert_{\infty} ^2 \ex{2 \mu T}}\Bigr\}
\\
&=& 3  \exp\Bigl\{-\frac { \mu T    r^2}{2C \Vert  \xi \Vert_{\infty} ^2 }\Bigr\}.
\end{eqnarray*}
Hence we have proved the following result.

\begin{theorem}\label{thm-exp-1}
There exists a  constant $c>0$ such that  for every process $\xi$ satisfying assumption $(\mathbf{S})$,
the following inequality holds
\begin{equation}\label{E122}
\ds{
\mathbb{P}\Bigl(\, \vert  u(t)\vert_X \ge r\Bigr)
\leq  3  \ex{-\frac { \mu t    r^2}{2c \Vert  \xi \Vert_{\infty} ^2 }}, \;\;r>0, \;\; t\geq 0.}
\end{equation}
\end{theorem}

The above theorem is also true for the stochastic Navier-Stokes equations with bounded multiplicative noise
 \begin{equation}
\label{eqn-SNSE}
du(t)+\rA u(t)+\rB(u(t),u(t))=g(u)\,dW(t),\ \ \ \ u(0)=u_0.
\end{equation}
Indeed,  the argument leading to Theorem 1.1 in \cite{Brz+Peszat_2000} remains unchanged in almost all steps. Here the Banach space $X$ is simply the Hilbert space $\rH$ and we assume that
\[
g:\rH \to M(\rK,\rH)
\]
is a continuous and bounded map.

Indeed,  these equations can be written in the form \eqref{E121} with time dependent operator $A(t)$ being dissipative for each $t$, i.e.

\begin{equation}\label{E127}
\d u+A(t)u(t)\d t=\xi (t)\d W(t), \qquad u(0)=0.
\end{equation}
with
\begin{eqnarray*}
A(t)v&=&Av+B(u(t),v)=Av+ \pi \big( u(t)\cdot \nabla v\big),\;\; t\geq 0\\
\xi(t)&=& g(u(t)) ,\;\; t\geq 0.
\end{eqnarray*}
Moreover, our proof is de facto simpler, as the solution $u$ is now a strong solution and one doesn't need to perform the approximation alluded to in \cite{Brz+Peszat_2000} (which could be done as in \cite{Brz+M+S_2005}).
Let $\mu>0$ be such that $-A+\mu I$ is dissipative on $\rH$. Denote, as before
\[
y(t)=
\ex{\mu t} u(t), \;\; t \geq 0.
\]

By an easy application of an  It\^o formula, we have
\begin{eqnarray*}
dy(t)&=&\mu \ex{\mu t} u(t) \, dt+ \ex{\mu t} \, d u(t)
\\
&=&\mu  y(t) \, dt+ \ex{\mu t} \, \big[ -A(t)u(t)\d t+\xi (t)\d W(t)  \big] \\
&=& \mu  y(t) \, dt -A(t) [ \ex{\mu t} \,  u(t) ] \d t+\ex{\mu t} \xi (t)\d W(t).
\end{eqnarray*}
Hence,
\[
dy(t)+ (A-\mu I) y(t) \d t= \eta(t)\d W(t).
\]
where
\[
\eta(t)=\ex{\mu t} \xi (t), \;\; t\geq 0.
\]

The crucial part of the argument   in this new case is the following one. If
\[
f_{\lambda}:X\ni x\mapsto\left(1+\lambda\vert
x\vert ^p\right)^{1/p}\in \mathbb{R},\]
we apply It\^o's formula to the function
$f_{\lambda}$ and process $y$ and we obtain
\begin{eqnarray}\label{E132}
&&f_{\lambda}\left(y(t)\right)
=f_{\lambda}\left(y(0)\right)+
\int^t_0f_{\lambda}^\prime\left(y(s)\right)\big((\mu I-A(s))y(s)\big)\d s
\\ \nonumber
&&\qquad +\int^t_0 f_{\lambda}^\prime \left(y(s)\right)
\left(\xi (s)\right)\d W(s)+\frac 12\int^t_0
{\rm tr}\,f_{\lambda}^{\prime\prime}\left(y(s)\right)
\left(\xi (s),\xi (s)\right)\d s.
\end{eqnarray}
Since $-A(s)$ is dissipative, the second term on the RHS of
\eqref{E132} is nonpositive.  This can also be seen from the formula
\[
f_{\lambda}^\prime(x)(y)= \lambda \big( 1 +\lambda \vert x \vert^p\big)^{\frac1p-1} \vert x \vert^{p-1} \langle x,y \rangle
\]
Therefore, using the properties of $B$, i.e. $\langle y(s), B(y(s),y(s)) \rangle=0$,  we get
\begin{eqnarray*}
f_{\lambda}^\prime\left(y\right)(\mu y -A(s)y)&=&f_{\lambda}^\prime\left(y\right)(\mu y -Ay +B(y,y))
\\
&=&
\lambda \big( 1 +\lambda \vert y \vert^p\big)^{\frac1p-1} \vert y \vert^{p-1} \langle y, -Ay +B(y,y) \rangle
\\
&=&
\lambda \big( 1 +\lambda \vert y \vert^p\big)^{\frac1p-1} \vert y \vert^{p-1} \Big[ \langle y, \mu y -Ay  \rangle +\langle y, B(y,y) \rangle \Big]
\\
&=&
\lambda \big( 1 +\lambda \vert y \vert^p\big)^{\frac1p-1} \vert y \vert^{p-1} \langle y, \mu y - Ay  \rangle \leq 0.
\end{eqnarray*}
Since also $f_{\lambda}(0)=1$,  we have
$$
f_{\lambda}\left(y(t)\right)
\le 1+\int^t_
0f_{\lambda}^\prime \left(y(s)\right) \left(\xi (s)\right) \d W(s)
+\frac{1}{2}
\int^t_0 {\rm tr}\, f_{\lambda}^{\prime\prime}\left(y(s)\right
)\left(\xi (s),\xi (s)\right)\d s.
$$
with,  because of Assumption $(\mathbf{H})$, $f_{\lambda}^\prime$ and $f_{\lambda}^{\prime\prime}$ satisfying
\begin{equation}\label{E31}
\vert f_{\lambda}^\prime(x)\vert
\le C_1\lambda^{1/p}
\qquad {\rm and}\qquad
\vert f_{\lambda}^{\prime\prime}(x)\vert \le C_2
\lambda^{2/p}.
\end{equation}
Arguing as in the proof of Theorem 1.2 in  \cite{Brz+Peszat_2000} we infer that
there exists an universal constant $c>0$ such that for all $T>0$ and $R>0$
\[\mathbb{P}\Bigl(\, \sup_{0\le t\le T}\vert  y(t)\vert_X \ge R\Bigr)
\le 3 \exp\Bigl\{-\frac {R^2}{4cM_T}\Bigr\}
\]
where
\[
M_T:= \esssup_{\Omega} \int^T_0\Vert \eta (t)\Vert ^2\d t.\]
Then, by repeating the proof of Theorem \ref{thm-exp-1}, we arrive at
 \begin{theorem}\label{thm-exp-NSEs}
There exists a  generic constant $c>0$ such that  the solution $u$ of the stochastic NSEs  \eqref{eqn-SNSE} the following inequality holds
\begin{equation}\label{eqn-exp-NSEs}
\mathbb{P}\Bigl(\, \vert  u(t)\vert_X \ge r\Bigr)
\leq  3  \ex{-\frac { \mu t    r^2}{2c \Vert  \xi \Vert_{\infty} ^2 }},  \;\;\; r>0,  t\geq 0.
\end{equation}
\end{theorem}

\section{Behavior of the solutions of the Navier-Stokes equations for large negative times}\label{section-negative}

In our paper with Mark Freidlin \cite{BCF_2013} we proved the following two results, see Propositions A.1 and A.2, for the general 2-D Navier Stokes Equations. We formulate them in way that does not need to use special notation used by us.

\begin{proposition}\label{prop_BCF_2013-infty}
Assume that $z \in C((-\infty,0];\rH)$ is such that
\begin{equation}
\label{eqn-B01} \lim_{t\to-\infty}\vert z(t)\vert_{\rH}=0
\end{equation}
and $S_{-\infty}(z)<\infty$, i.e.
\begin{equation}
\label{eqn-B02}
\int_{-\infty}^0\vert  z^\prime(t)+\rA z(t)+\rB(z(t),z(t)) \,\vert_{\rH}^2\, dt <\infty.
  \end{equation}
  Then, we have $z(0) \in\,V$,
 \begin{equation}
\label{eqn-A04}
\lim_{t\to-\infty} \vert z(t)\vert_{\rV}=0,
\end{equation}
and
\begin{equation}
\label{eqn-A10}
\int_{-\infty}^0 \vert Az(t)\vert_\rH^2\, dt  + \int_{-\infty}^0 \vert z^\prime(t)\vert_\rH^2\, dt<\infty.
\end{equation}
Moreover, there exists a continuous and strictly increasing function $\varphi: [0,\infty) \to [0,\infty)$ such that
$\varphi(0)=0$ and, if $z$ satisfying condition \eqref{eqn-B01}  is a solution to the problem
\begin{equation}
\label{eqn-B05}
z^\prime(t)+\rA z(t)+\rB(z(t),z(t))=f(t), \ \ \  t\leq 0,
\end{equation}
with  $f$ being an element of $L^2(-\infty,0;\rH)$, then
 \begin{equation}
\label{eqn-B06} \vert z(0) \vert_{\rV}^2+ \int_{-\infty}^0 \vert Az(t)\vert_\rH^2\, dt  + \int_{-\infty}^0 \vert z^\prime(t)\vert_\rH^2\, dt \leq
\varphi( \int_{-\infty}^0 \vert f(t)\vert_\rH^2\, dt).
\end{equation}
\end{proposition}

\begin{proposition} \label{prop_BCF_2013-inftybis}
Assume that $\alpha \in\,(0,1/2)$. If a function  $z \in\,C((-\infty,0];\rH)$, satisfying condition \eqref{eqn-B01}, satisfies also
  \begin{equation}
\label{eqn-B03}
 \int_{-\infty}^0\vert  z^\prime(t)+\rA z(t)+\rB(z(t),z(t)) \,\vert_{D(\rA^{\frac\a 2})}^2\, dt <\infty,
  \end{equation}
we have \begin{equation}
\label{eqn-B04} z(0)\in D(\rA^{\frac{\a}2+\frac12}),
  \end{equation}
\begin{equation}
\label{eqn-A04'}
\lim_{t\to-\infty} \vert z(t)\vert_{D(\rA^{\frac{\a}2+\frac12})}=0,
\end{equation}
and
\begin{equation}
\label{eqn-A11}
\int_{-\infty}^0 \vert A^{\frac{\a}2+1} z(t)\vert_\rH^2\, dt  + \int_{-\infty}^0 \vert  A^{\frac{\a}2}  z^\prime(t)\vert_\rH^2\, dt<\infty.
\end{equation}
Moreover, there exists a continuous and strictly increasing function $ \varphi_\alpha: [0,\infty) \to [0,\infty)$ such that
$ \varphi_\alpha(0)=0$ and if
$z $, satisfying condition \eqref{eqn-B01}, is a solution to  problem \eqref{eqn-B05} with  $f \in L^2(-\infty,0;D(\rA^{\frac\a 2}))$, then

 \begin{eqnarray}
\label{eqn-B10} \vert z(0) \vert_{D(\rA^{\frac{\a}2+\frac12})}^2&+& \int_{-\infty}^0 \vert A^{\frac{\a}2+1} z(t)\vert_\rH^2\, dt
 + \int_{-\infty}^0 \vert  A^{\frac{\a}2}  z^\prime(t)\vert_\rH^2\, dt
  \\
  &\leq& \varphi_\alpha( \vert f\vert^2_{L^2(-\infty,0);D(\rA^{\frac\a 2})}).
\label{eqn-B11}
\end{eqnarray}
\end{proposition}
The reason for the restriction $\alpha \in\,(0,1/2)$ in Proposition \ref{prop_BCF_2013-inftybis} lies in the fact that we have used continuity of the Leray-Helmhotz projection $P$ from
$\rH^{\alpha}(\mathcal{O},\mathbb{R}^2)$ into  $ D(\rA^{\alpha/2})$, see Propossition 2.1 in \cite{BCF_2013} (and \eqref{eqn-Leray-Helmholtz} in the current paper).

The aim of this section is show that in the case of the 2-D NSEs with periodic boundary conditions, i.e. NSEs on a 2-dimensional torus, Proposition \ref{prop_BCF_2013-inftybis} holds true for any $\alpha>0$. Of course we will only need to consider the case $\alpha \geq \frac12 $. The main result in this section is as follows.

\begin{proposition} \label{prop_BCF_2013-inftythree}
Assume that $\alpha>0$. If $z$ satisfies  conditions \eqref{eqn-B01} and
\eqref{eqn-B03}, then it satisfies  \eqref{eqn-B04},
 \eqref{eqn-A04'}, and \eqref{eqn-A11}  as well.\\
Moreover, there exists a continuous and strictly increasing function $ \varphi_\alpha: [0,\infty) \to [0,\infty)$ such that
$ \varphi_\alpha(0)=0$ and if
$z $, satisfying condition \eqref{eqn-B01}, is a solution to  problem \eqref{eqn-B05}, with  $f \in L^2(-\infty,0;D(\rA^{\frac\a 2}))$, then inequality \eqref{eqn-B10} holds as well.
\end{proposition}

The following proof is an adaptation of the proof of
Proposition 2.2 from \cite{BCF_2013}. In fact, we follow the lines quite literary. As mentioned earlier, we only need to consider the case $\a \in\,[1/2,\infty)$.  Since  the estimates for the nonlinear term $B$ given in Propositions \ref{prop-Leray-fractional-alpha} and \ref{prop-Leray-fractional-alpha2}  are different for $\alpha \leq 1$ and $\alpha >1$ we will have to consider two cases: $\a \in\,[1/2,1]$ and $\a \in\,(1, \infty)$.  If $\a \in\,[1/2,1]$ then  we can  use inequality \eqref{ineqn-B-fractional} with $s=2$. In this case the proof from \cite{BCF_2013} is virtually the same. Note however, that if $\alpha>1$ and inequality \eqref{ineqn-B-fractional2} holds, this does not imply that inequality \eqref{ineqn-B-fractional} with $s=2$ holds.

\begin{proof}[Proof of Proposition \ref{prop_BCF_2013-inftythree}]
In the whole proof all the norms and scalar products are in $\rH$.

For the readers convenience and the completeness of the results we will  prove our result in the special  case

\[ \alpha=1\]

So, let us fix  $\phi \in\,D(\rA)$ and a function $z$ satisfying conditions   \eqref{eqn-B01} and \eqref{eqn-B03}. Following the methods from the proof of Proposition 2.1 from \cite{BCF_2013}  it is sufficient to prove \eqref{eqn-B04},
 \eqref{eqn-A04'}, and \eqref{eqn-A11}.

Since the function $z$ satisfies inequality \eqref{eqn-B06},   we can find a decreasing sequence $\{s_n\}$ such that $s_n\searrow -\infty$,  $z(s_n) \in D(\rA)$, $n\in \mathbb{N}$ and \comad{Here we take $\alpha=1$ so that $\frac{\alpha+1}2=1$.}
\begin{equation}\label{eqn-NSE02''}
\lim_{n\to\infty} \vert \rA z(s_n)\vert_\rH=0.
\end{equation}
Arguing as in the proof of  \cite[Proposition 3.3]{BCF_2013}, 
we infer that the function  $\vert \rA z(\cdot)\vert_{\rH}^2 $ is absolutely continuous and satisfies the following identity on $(-\infty,0]$
\begin{equation}\label{eqn-NSE03c}
\frac12 \frac{d}{dt}\vert \rA z(t) \vert_\rH^2+\vert \rA^{\frac{3}2} z(t) \vert_\rH^2=-(B(z(t),z(t)),\rA^{2}z(t))_\rH+(f(t),\rA^2 z(t))_\rH.
\end{equation}
 In view of  inequality \eqref{ineqn-B-fractional}, with $s=2$ and $\alpha =1$,  we  infer that there exists $c>0$ such
 the following
inequality is satisfied for $f\in D(\rA^{\frac{1}2}) $
\begin{eqnarray*}
-(B(z,z),\rA^{2}z)+(\rA f,\rA^{2} z)&=& -(\rA^{\frac12}B(z,z),\rA^{\frac32}z)+(\rA^{\frac12}f,\rA^{\frac32} z)\\
 &\leq& \frac12 \vert \rA^{\frac{3}2} z \vert^2+c\,\vert \rA z \vert^4+  \vert \rA^{\frac{1}2} f \vert^2, \;\;z\in D(\rA).
\end{eqnarray*}

Hence, we infer that on $(-\infty,0]$, we have
\begin{equation}\label{eqn-NSE04'''}
 \frac{d}{dt}\vert \rA z(t) \vert_\rH^2 +\vert \rA^{\frac{3}2} z(t) \vert_\rH^2
\leq c\,\vert \rA z(t) \vert_\rH^2 \vert \rA z(t) \vert_\rH^2  +2\vert \rA^{\frac{1}2} f(t) \vert_\rH^2.
\end{equation}

Therefore,   by the Gronwall Lemma,  we get
\begin{eqnarray}\label{eqn-NSE05''}
\vert \rA z(t)\vert_\rH^2
&\leq&  \vert \rA z(s)\vert_\rH^2 \ex{c \int_s^t \vert\rA z(r)\vert_\rH^2\,dr}\\
&+&2\int_s^t \vert \rA^{\frac{1}2} f(r)\vert_\rH^2 \ex{c\int_r^t \vert \rA z(\rho)\vert_\rH^2\, d\rho}\, dr,\;\; -\infty < s\leq t\leq 0.
\nonumber
\end{eqnarray}
Note that  by inequality \eqref{eqn-B06} 
we have
\[
\sup_{n\geq 1}\int_{s_n}^t \vert \rA z(r)\vert_\rH^2\, dr
\leq \varphi(\vert f\vert^2)
\delc{\frac{2\,C_2}{\lambda_1^2} \ex{\frac {C_2}{\lambda_1^2} \vert f\vert^4} \vert f\vert^6+2 \vert f\vert^2 <\infty} ,\;\;\; t \leq 0.
\]
where  for the sake of brevity, we set $\vert f\vert_{L^2(-\infty,0;\rH)}=\vert f\vert$.

Hence, using  inequality \eqref{eqn-NSE05''} with $s=s_n$ from \eqref{eqn-NSE02''} and then taking the limit as $n\to\infty$,  we infer that
\begin{equation}\label{eqn-NSE06''}
\vert \rA z(t)\vert_\rH^2
\leq   2\int_{-\infty}^t \vert \rA^{\frac{1}2} f(r)\vert_\rH^2 \ex{c\int_r^t\vert \rA z(\rho)\vert_\rH^2\, d\rho}\, dr,\ \ \  t\leq 0.
\end{equation}
 Therefore,
 we  conclude that
\begin{equation}\label{eqn-A16'}
\sup_{t\leq 0} \vert \rA z(t)\vert^2 \leq
2 \ex{C  \varphi(\vert f\vert^2}) \int_{-\infty}^0 \vert \rA^{\frac12} f(r)\vert_\rH^2\,dr  .
\end{equation}
Moreover, by assumption \eqref{eqn-B03}, definition  \eqref{eqn-B05} of the function $f$ and inequality \eqref{eqn-B06} we have
\[\int_{-\infty}^0 \vert \rA^{\frac{1}2} f(r)\vert^2 \ex{c\int_r^0 \vert \rA z(\rho)\vert_\rH^2\, d\rho}\, dr
\leq  \ex{c\int_{-\infty}^0 \vert \rA z(\rho)\vert_\rH^2\, d\rho} \int_{-\infty}^0 \vert \rA^{\frac{1}2} f(r)\vert^2
\, dr
<\infty,\]
we infer that
\[\lim_{t\to-\infty} \int_{-\infty}^t \vert \rA^{\frac{1}2} f(r)\vert_\rH^2 \exp\ex{c\int_r^t\vert \rA z(\rho)\vert_\rH^2\, d\rho}\, dr=0.\]
Hence, due to \eqref{eqn-NSE06''}, we have that  \eqref{eqn-A04'} holds for $\alpha=1$, i.e.
\begin{equation*}
\lim_{t\to-\infty} \vert z(t)\vert_{D(\rA)}^2=0.
\end{equation*}

Now, let us prove the first one of inequalities  \eqref{eqn-A11}, with $\alpha=1$.
For this aim, let us   observe that from \eqref{eqn-NSE04'''} we  deduce that
\begin{eqnarray*}
\vert \rA z(0)\vert_\rH^2&+&\int_{-\infty}^0 \vert \rA^{\frac{3}2} z(t)\vert_\rH^2\, dt
\leq c\int_{-\infty}^0  \vert \rA z(t)\vert_\rH^4\,dt
+2\int_{-\infty}^0\vert \rA^{\frac{1}2} f(t)\vert_\rH^2 \,dt
\\&\leq&
c\sup_{t \leq 0} \vert \rA z(t)\vert_\rH^2 \int_{-\infty}^0   \vert \rA z(t)\vert_\rH^2\,dt
+2\int_{-\infty}^0\vert \rA^{\frac{1}2} f(t)\vert_\rH^2 \,dt.
\end{eqnarray*}

Taking into account inequalities \eqref{eqn-A16'} and   \eqref{eqn-B06}, 
we infer that
\begin{eqnarray}\label{eqn-NSE08'}
\vert \rA z(0)\vert_\rH^2&+&\int_{-\infty}^0 \vert \rA^{\frac{3}2} z(t)\vert_\rH^2\, dt
\leq 2\int_{-\infty}^0\vert \rA^{\frac{1}2} f(t)\vert_\rH^2 \,dt \\ &+&
2c \varphi(\vert f\vert^2)
 \ex{c  \varphi(\vert f\vert^2)}  \int_{-\infty}^0 \vert \rA^{\frac{1}2} f(r)\vert_\rH^2\,dr,
\nonumber
\end{eqnarray}
and this  concludes the proof of the first part of inequalities  \eqref{eqn-A11}
.

In order to prove the second  of inequalities \eqref{eqn-A11} (with $\alpha=1$), i.e.
\[
\int_{-\infty}^0 \vert  \rA^{\frac 12}  z^\prime(t)\vert_\rH^2\, dt<\infty\]
by the maximal regularity of the linear Stokes problem, it is enough to show that
\[ \int_{-\infty}^0 \vert \rA^{\frac{1}2}   B(z(t),z(t))\vert_\rH^2\, dt < \infty.\]

According to inequality \eqref{ineqn-B-fractional} (with $s=2$ and $\alpha=1$), we get, similarly to  \eqref{eqn-NSE08'}, the following estimate
\begin{eqnarray}\label{eqn-B20}
\int_{-\infty}^0 \vert \rA^{\frac{1}2}   B(z(t),z(t))\vert_\rH^2\, dt &\leq&  C  \int_{-\infty}^0 \vert \rA z(t) \vert^4_\rH
  \, dt\\
 &\leq& 2c \varphi(\vert f\vert^2)
 \ex{c  \varphi(\vert f\vert^2)}  \int_{-\infty}^0 \vert \rA^{\frac{1}2} f(r)\vert_\rH^2\,dr .
\nonumber
\end{eqnarray}
The proof, for $\alpha=1$, is now complete.

The case $\alpha >1$ has to be  treated  very carefully. To this purpose, we first consider the case $\alpha \in (1,2]$.
We fix  $\phi \in\,D(\rA^{\frac{\alpha+1}2})$ and a function $z \in\,C((-\infty,0];\rH)$, such that $z(0)=\phi$, satisfying conditions \eqref{eqn-B01} and \eqref{eqn-B03}, i.e.
\[ \lim_{t\to-\infty}\vert z(t)\vert_{\rH}=0,\]
and
\[ \int_{-\infty}^0 |f(t)|^2_{D(\rA^{\frac\a 2})}dt:=\int_{-\infty}^0\vert  z^\prime(t)+\rA z(t)+\rB(z(t),z(t)) \,\vert_{D(\rA^{\frac\a 2})}^2\, dt <\infty.\]
  Since the assumptions in the present proposition
 are stronger than the assumptions of Propositions \ref{prop_BCF_2013-infty} and \ref{prop_BCF_2013-inftybis}, we can freely use the results from their proofs, see \cite{BCF_2013}.

As before, it is sufficient to prove that $z$ satisfies conditions \eqref{eqn-B04},
 \eqref{eqn-A04'}, and \eqref{eqn-A11}.
We notice that, due to inequality  \eqref{eqn-A11} with $\alpha=1$, 
we can find a decreasing sequence $\{s_n\}$ such that $s_n\todown -\infty$ and
\begin{equation}\label{eqn-NSE02''d}
\lim_{n\to\infty} \vert \rA^{\frac{3}2} z(s_n)\vert_\rH=0.
\end{equation}
Hence, as $\alpha \leq 2$, we get \comcd{For the below it is sufficient to assume that $\alpha \leq 2$.}
\[
\lim_{n\to\infty} \vert \rA^{\frac{\alpha+1}2} z(s_n)\vert_\rH=0.
\]
  Therefore we can deduce that the function   $\vert \rA^{\frac{\a+1}2} u(t)\vert_{\rH}^2 $  is absolutely continuous and satisfies the following identity on $(-\infty,0]$
\begin{equation}
\label{eqn-NSE03d}
\frac12 \frac{d}{dt}\vert \rA^{\frac{\alpha+1}2} z(t) \vert_{\rH}^2+\vert \rA^{\frac{\a}2+1} z(t) \vert_{\rH}^2=-(B(z(t),z(t)),\rA^{{\a}+1}z(t))+(f(t),\rA^{\a+1} z(t))
\end{equation}
By inequality \eqref{ineqn-B-fractional2}, since $\frac\alpha2 \leq \frac32$, we infer that
\[\begin{array}{l}
\ds{-(B(z,z),\rA^{{\a}+1}z)=(\rA^{\frac\alpha2}B(z,z),\rA^{1+\frac\alpha2}z)
\leq c \vert \rA^{\frac\alpha2}z \vert_{\rH}  \vert \rA^{\frac{\alpha+1}2}z \vert_{\rH}
\vert \rA^{1+\frac\alpha2}z \vert_{\rH}}\\
\vs\\
\ds{\leq  \frac14 \vert \rA^{1+\frac\alpha2}z \vert_{\rH}^2 +c \vert \rA^{\frac\alpha2}z \vert_{\rH}^2  \vert \rA^{\frac{\alpha+1}2}z \vert_{\rH}^2\leq  \frac14 \vert \rA^{1+\frac\alpha2}z \vert_{\rH}^2 +c \vert \rA^{\frac32}z \vert_{\rH}^2  \vert \rA^{\frac{\alpha+1}2}z \vert_{\rH}^2.}
\end{array}\]

Due to \eqref{eqn-NSE03d}, this implies that
\[ \frac{d}{dt}\vert \rA^{\frac{\alpha+1}2} z(t) \vert_{\rH}^2 +\vert \rA^{\frac{\a}2+1} z(t) \vert_\rH^2
\leq c\vert \rA^{\frac32} z(t) \vert_\rH^2 \vert \rA^{\frac{\alpha+1}2} z(t) \vert_\rH^2  +2\vert \rA^{\frac{\a}2} f(t) \vert_\rH^2.
\]
Therefore,   by the Gronwall Lemma, for any $-\infty < s\leq t\leq 0$ we get
\begin{eqnarray}\label{eqn-NSE05''d}
\vert \rA^{\frac{\alpha+1}2}z(t)\vert_\rH^2
&\leq&  \vert \rA^{\frac{\alpha+1}2}z(s)\vert_\rH^2 \exp\left(c \int_s^t \vert\rA^{\frac32} z(r)\vert_\rH^2\,dr\right)\\
&+&2\int_s^t \vert \rA^{\frac{\a}2} f(r)\vert_\rH^2 \exp\left(c\int_r^t \vert \rA^{\frac32} z(\rho)\vert_\rH^2\, d\rho\right)\, dr.
\nonumber
\end{eqnarray}
Note that  by inequality  \eqref{eqn-B11} with $\alpha=1$
we have
\[
\sup_{n\geq 1}\int_{s_n}^t \vert \rA^{\frac32} z(r)\vert_\rH^2\, dr  \leq
\varphi_1(\vert f\vert_{\frac 12}^2 ),
\]
where
 we use notation shortcut  $\vert f\vert_{\alpha/2}=\vert f\vert_{L^2(-\infty,0);D(\rA^{\frac\a 2})}$.
Hence, using  inequality \eqref{eqn-NSE05''d} with $s=s_n$ from \eqref{eqn-NSE02''d} and then taking the limit as $n\to\infty$,  we infer that

\begin{equation}\label{eqn-NSE06''d}
\vert \rA^{\frac{\alpha+1}2}z(t)\vert_\rH^2
\leq   2\int_{-\infty}^t \vert \rA^{\frac{\a}2} f(r)\vert_\rH^2 \ex{c\int_r^t\vert \rA^{\frac32} z(\rho)\vert_\rH^2\, d\rho}\, dr,\ \ \  t\leq 0,
\end{equation}
so that
\begin{equation}\label{eqn-A16'd}
\sup_{t\leq 0} \vert \rA^{\frac{\alpha+1}2} z(t)\vert \leq
2\int_{-\infty}^0 \vert \rA^{\frac{\a}2} f(r)\vert_\rH^2\,dr \ex{ \varphi_1(\vert f\vert_{\frac 12}^2 )}=2 \vert f \vert_{\frac\alpha2}^2 \ex{ \varphi_1(\vert f\vert_{\frac 12}^2 )}.
\end{equation}
Moreover, as
\[\int_{-\infty}^0 \vert \rA^{\frac{\a}2} f(r)\vert_\rH^2 \ex{c\int_r^t \vert \rA^{\frac32} z(\rho)\vert_\rH^2\, d\rho}\, dr<\infty,\]
we have that
\[\lim_{t\to-\infty} \int_{-\infty}^t \vert \rA^{\frac{\a}2} f(r)\vert_\rH^2 \ex{c\int_r^t\vert \rA^{\frac32} z(\rho)\vert_\rH^2\, d\rho}\, dr=0,\]
and  \eqref{eqn-A04'}  follows from \eqref{eqn-NSE06''d}.


Next, we observe that, due to  \eqref{eqn-NSE04'''},

\[\begin{array}{l}
\ds{\vert \rA^{\frac{\a+1}2} z(0)\vert_\rH^2+\int_{-\infty}^0 \vert \rA^{\frac{\a+2}2} z(t)\vert_\rH^2\, dt
\leq c\int_{-\infty}^0  \vert \rA^{\frac32} z(t)\vert^2 \vert \rA^{\frac{\a+1}2} z(t)\vert_\rH^2\,dt}\\
\vs\\
\ds{
+2\int_{-\infty}^0\vert \rA^{\frac{\a}2} f(t)\vert_\rH^2 \,dt
\leq c \sup_{t\leq 0} \vert \rA^{\frac{\a+1}2} z(t)\vert_\rH^2\int_{-\infty}^0  \vert \rA^{\frac32} z(t)\vert_\rH^2 \,dt
+2 \vert f \vert_{\frac\alpha2}^2.}
\end{array}\]
Taking into account inequalities \eqref{eqn-A16'd} and  \eqref{eqn-A11} with $\alpha=1$,
we infer that
\begin{eqnarray}\label{eqn-NSE08'd}
\vert \rA^{\frac{\a+1}2} z(0)\vert_\rH^2+\int_{-\infty}^0 \vert \rA^{\frac{\a+2}2} z(t)\vert_\rH^2\, dt
&\leq& 2 \,\vert f \vert_{\frac\alpha2}^2
\\
&+& c \vert f \vert_{\frac\alpha2}^2 \ex{ \varphi_1(\vert f\vert_{\frac 12}^2 )} \varphi_{1}(\vert f \vert_{\frac12}^2)
\nonumber
\end{eqnarray}
and this  concludes the proof of the first part of inequality \eqref{eqn-A11}.

Invoking the maximal regularity of the Stokes evolution equation, in order to prove  the second inequality in \eqref{eqn-A11},
it is enough to show that
\[ \int_{-\infty}^0 \vert \rA^{\frac{\a}2}   B(z(t),z(t))\vert_\rH^2\, dt < \infty.\]

According to inequalities \eqref{ineqn-B-fractional2},
\eqref{eqn-A16'd} and  \eqref{eqn-A11} with $\alpha=1$,
we have

\begin{eqnarray}
\int_{-\infty}^0 \vert \rA^{\frac{\a}2}   B(z(t),z(t))\vert_\rH^2\, dt &\leq & c  \int_{-\infty}^0 \vert \rA^{\frac{\a}2} z(t) \vert^2_\rH
\vert \rA^{\frac{\a+1}2} z(t) \vert_\rH^2  \, dt\\
&\leq & c \sup_{t\leq 0}  \vert \rA^{\frac{\a+1}2} z(t) \vert^2_\rH
\int_{-\infty}^0  \vert \rA^{\frac32} z(t) \vert_\rH^2 \, dt
\nonumber\\
&\leq & \vert f \vert_{\frac\alpha2}^2 \ex{ \varphi(\vert f\vert_{\alpha/2}^2 )} \varphi_{1}(\vert f \vert_{\frac12}^2).
\nonumber
\end{eqnarray}
The proof in the case $\alpha \in (1,2]$ is now complete. A simple extension of the last argument and mathematical induction with respect to the integer part of $\alpha$ can provide a complete proof for all $\alpha \geq 0$.

\end{proof}

\end{document}